\documentclass[times,11pt,reqno]{amsproc}

\usepackage[top=3cm, bottom=2cm, outer=2cm, inner=2cm, heightrounded,marginparwidth=3cm, marginparsep=0cm]{geometry}
\usepackage{amsmath}
\usepackage{amssymb}
\usepackage{amsthm}
\usepackage{amsfonts,mathrsfs}

\usepackage{hyperref}
\usepackage{booktabs}
\usepackage{array,enumitem}

\usepackage{tikz}
\usetikzlibrary{arrows.meta, math,positioning, calc}

\usepackage{xcolor}

\newtheorem{theorem}{Theorem}[section]
\newtheorem{proposition}[theorem]{Proposition}
\newtheorem{lemma}[theorem]{Lemma}
\newtheorem{corollary}[theorem]{Corollary}

\newtheorem*{theorem*}{Theorem}

\theoremstyle{definition}
\newtheorem{definition}[theorem]{Definition}

\newtheorem{example}[theorem]{Example}
\newtheorem{remark}[theorem]{Remark}

\numberwithin{equation}{section}

\def\R#1{\mathbb{R}^{#1}}

\def\scal#1#2{\langle #1; #2 \rangle}

\newcommand\fie{\Bbbk}

\def\alg{\mathbb{A}}

\def\valg{\mathbb{V}}
\def\ualg{\mathbb{U}}

\newcommand{\fusion}{\mathscr{F}}

\newcommand{\Id}{\operatorname{Id}}

\newcommand{\dum}{\,\cdot\,\,}

\DeclareMathOperator{\tr}{tr}

\DeclareMathOperator{\Aut}{Aut}

\DeclareMathOperator{\spec}{spec}

\DeclareMathOperator{\supp}{supp}
\DeclareMathOperator{\Spec}{Spec}

\DeclareMathOperator{\St}{St}
\DeclareMathOperator{\Ax}{Ax}
\DeclareMathOperator{\Sym}{Sym}

\def\ualg{\mathbb{U}}
\def\malg{\mathbb{M}}
\newcommand{\vm}[1]{\malg(#1)}
\newcommand{\refl}{h}

\begin{document}

\title[Peirce spectral rigidity of decorated incidence algebras]
 {On the Peirce spectral rigidity of decorated incidence algebras}

\author[D.~J.~F.~Fox]{Daniel J. F. Fox}
\address{Departamento de Matem\'atica Aplicada, Escuela T\'ecnica Superior de Arquitectura, Universidad Polit\'ecnica de Madrid, Av. Juan de Herrera 4, 28040 Madrid, Spain}
\email{daniel.fox@upm.es}

\author[V.~G.~Tkachev]{Vladimir G. Tkachev}
\address{Department of Mathematics, Link\"oping University, 58183, Sweden}
\email{vladimir.tkatjev@liu.se}
\thanks{V.G.T. has been supported by the Stiftelsen L\"angmanska kulturfonden, Grant BA26-2924.}

%
\date{\today}

\begin{abstract}
The multiplication operator of an idempotent in a metrized commutative nonassociative algebra is a self-adjoint linear endomorphism whose spectrum reflects the algebraic and geometric structure encoded in the algebra. Although for general algebras the spectra of idempotents can behave arbitrarily, for the most interesting classes -- among them algebras of Clifford and Jordan type arising in geometry -- the spectrum is highly constrained. We consider a family of algebras determined by a partial Steiner triple system (PSTS) with blocks decorated by signs, each block determining an idempotent; even for the simplest PSTS the resulting algebras can be quite complicated. For decorations of the Grassmannian PSTS $G_2(m)$, whose blocks are $3$-subsets of an $m$-element set and whose points parametrize an underlying incidence geometry, the spectra behave rigidly: the spectrum with multiplicity of a block idempotent is a function of a single combinatorial parameter, the number of negatively signed Pasch configurations -- quadrilateral subconfigurations -- containing the block. The resulting block spectra organize into a spectral profile, recovered by binomial inversion from a hierarchy of moment invariants. For a distinguished family of decorations, the associated algebras are axial algebras satisfying a fusion law independent of $m$, and for these algebras there is given an alternative explicit matrix model exhibiting the rigidity directly; the even part of this model is a direct sum containing a polar algebra summand, linking the construction to polar algebras and symmetric Clifford systems.
\end{abstract}

\maketitle

\section{Introduction}

\subsection{Background and motivation}
The multiplication operator $L(c):x\mapsto c\circ x$ of an idempotent $c$ in a metrized commutative nonassociative algebra is a self-adjoint linear endomorphism of the underlying vector space, and its spectrum is accordingly a natural linear-algebraic invariant attached to the idempotent. When the structure of the algebra is determined by some underlying combinatorial data, the spectra of its idempotents can be constrained in ways that reflect the combinatorial data defining the algebra. The present paper studies this phenomenon for a family of commutative algebras built from partial Steiner triple systems.

A \emph{partial Steiner triple system} (PSTS) is a pair
$G=(X,\mathcal B),$ where $X$ is a finite set and $\mathcal B$ is a family of $3$-element
subsets of $X$, called \emph{blocks}, such that any two distinct points of
$X$ belong to at most one block (if any two distinct points of $X$ belong to exactly one block, $G$ is called a Steiner triple system, abbreviated STS). The cardinality $n=|X|$ is called the order of the system.
We say that $i$ and $j$ are \emph{collinear}, and write $ i\sim j$, if there is a block containing $i$ and $j$. In this case the unique point $k\in X$ such that $ \{i,j,k\}\in\mathcal B,$ is written $$ i\wedge j=k. $$
A \emph{Pasch configuration} (also called a \emph{quadrilateral}) in a PSTS $G=(X,\mathcal B)$ is a set of $4$ blocks as in \eqref{pasch} whose union consists of $6$ points
\begin{align}\label{pasch}
&\{\{i, j, k\}, \{i, b, c\},\{a, j, c\},\{a, b, k\}\}\subseteq\mathcal B.
\end{align}

Given a field $\fie$, to a PSTS $G = (X, \mathcal B)$ and parameters $\alpha,\beta,\gamma\in\fie$, not all zero, we associate a commutative algebra with underlying vector space $\fie\{X\}$ and generated by the distinguished basis $\{e_{i}: i \in X\}$ subject to the relations (see \cite[Lemma $2.1$]{Fox2022})
\begin{equation}\label{eq:algebraSdef}
\begin{aligned}
e_i\circ e_i&=\gamma e_i,\\
e_i\circ e_j&=
\left\{
  \begin{array}{ll}
    \alpha(e_i+e_j)+\beta e_{i\wedge j},
    & \hbox{if $i\sim j$, $i\neq j$,}\\
    0,
    & \hbox{if $i\not\sim j$.}
  \end{array}
\right.
\end{aligned}
\end{equation}
Throughout the paper $\fie$ is supposed to have characteristic $0$.
For $\alpha=\gamma=0$ and $\beta=1$, the resulting algebra is denoted by $\alg(G,\fie,\circ)$ and is called the \emph{incidence algebra} of $G$. For a Steiner triple system, adding a unit to the incidence algebra yields a commutative nonassociative algebra studied by Mendelsohn \cite{Mendelsohn}. Other choices of the parameters recover several well-known algebras. The parameters $\alpha = 1/(1-n)$ and $\beta = 0$ yield the $n$-dimensional algebras on the standard representation of the symmetric group $S_{n+1}$ studied by Griess \cite{Dong-Griess} and Harada \cite{Harada}. The choice $ \beta=-\alpha$, $\gamma=1,$ yields the Matsuo algebras associated with partial Steiner triple systems such as Fischer spaces (up to the conventional normalization of the parameter $\alpha$) \cite{DeMedts-Rehren, Hall-Rehren-Shpectorov,Matsuo-3transposition}.

A more general class of algebras is obtained by allowing the parameters $\alpha$ and $\beta$ in \eqref{eq:algebraSdef} to depend on $i$ and $j$. Here we consider the particular case given by modifying the incidence algebra by choosing for every block a sign
$$
\delta(\{i,j,k\})\in\{\pm1\},
$$
and replacing the multiplication relation $e_i\circ e_j=e_{i\wedge j}$
by
$$
e_i\circ_\delta e_j
=
\delta(\{i,j,i\wedge j\})\,e_{i\wedge j}.
$$
The resulting algebra $\alg^\delta(G,\fie,\circ_\delta)$ is called the \textit{decorated incidence algebra}
associated with the pair $(G,\delta)$ comprising a PSTS with blocks labeled by signs. The usual incidence algebra of $G$ is recovered by using the trivial decoration in which all signs are $+$.

Motivation for studying decorated incidence algebras comes from the study of Hsiang algebras, a class of metrized commutative algebras introduced by the second author in relation to the classification of cubic minimal cones \cite{Nadirashvili-Tkachev-Vladuts, Fox-Tkachev} (in \cite{Nadirashvili-Tkachev-Vladuts} they are called \emph{REC algebras}). It is an empirical observation that for a certain important subclass of Hsiang algebras their cubic forms (given by metrically pairing an element with its square) are signed sums of monomials indexed by particular PSTSs which we call Hsiang PSTS (a direct combinatorial characterization of the Hsiang PSTSs is not yet known). The decorated incidence algebra construction yields a way of recovering the Hsiang algebra from the associated \emph{Hsiang PSTS}. For example, for different decorations applied to the PSTS $(9_2, 6_3)$ there result non-isomorphic $9$-dimensional algebras with cubic forms equal to the determinant and permanent of a $3\times 3$ matrix and the algebra associated with the determinant is Hsiang, as is described briefly in Example~ \ref{rem:detper}.

In general the decorated incidence algebras associated with a given $G$ and different decorations need not be isomorphic. The automorphism group of the PSTS $G$ acts on the decorations of $G$ by pullback, so acts on pairs $(G, \delta)$, and the associated decorated incidence algebras are isomorphic. However, distinct decorations of a fixed PSTS can determine isomorphic decorated incidence algebras, reflecting a kind of gauge symmetry of the construction. Given a
sign assignment
$$
\varepsilon:X\to\{\pm1\},
$$
one may rescale the basis vectors according to
$$
e_i\mapsto \varepsilon(i)e_i.
$$
This induces a transformation of decorations,
$$
\delta^\varepsilon(\{i,j,k\})
=
\varepsilon(i)\varepsilon(j)\varepsilon(k)\,
\delta(\{i,j,k\}).
$$
Two decorations related in this way are called \textit{gauge equivalent}. Gauge equivalent decorations give rise to isomorphic decorated incidence
algebras. Consequently, invariants of the decorated incidence algebra reflect properties of the isomorphism class of the PSTS and the gauge-equivalence class of the decorated PSTS. Note that it is not claimed that automorphisms of $G$ and gauge equivalencies generate the automorphism group of the decorated incidence algebra, although this is true for some examples.

For an operator $M$, let
$\spec(M)$ denote the set of distinct eigenvalues of $M$, and let
$\Spec(M)$ denote the multiset spectrum, where eigenvalues are counted
with algebraic multiplicity. Thus
$$
\Spec(M)=\{1,(-1)^2,0^3\}
$$
means that $1$, $-1$, and $0$ occur with multiplicities $1$, $2$, and
$3$, respectively.
The \emph{Peirce spectrum} of an idempotent $c$ in $\alg^\delta(G,\fie,\circ_\delta)$ is the multiset spectrum $\Spec(L(c))$ of its multiplication operator $L_{\delta}(c):x\mapsto c\circ_\delta x$.

For any PSTS and any block
$
b=\{i,j,k\}\in\mathcal B,
$
the element
\begin{align}\label{blockidempotent}
c_b^\delta
=
\frac12\,\delta(b)(e_i+e_j+e_k)
\end{align}
is an idempotent of the decorated incidence algebra $\alg^\delta(G,\fie,\circ_\delta)$ called the \emph{block idempotent} associated with $b$. The principal object of study here is the multiset Peirce spectrum
$
\Spec\bigl(L_\delta(c_b^\delta)\bigr)
$
of a block idempotent, $b\in\mathcal B$.
Whenever multiplicities are irrelevant, we write $\spec(L_\delta(c_b^\delta))$ for the corresponding set of distinct
eigenvalues.

In $\alg^\delta(G,\fie,\circ_\delta)$ there can be many idempotents besides the block idempotents and it need not be that an algebra automorphism of $\alg^\delta(G,\fie,\circ_\delta)$ sends block idempotents to block idempotents (Example~\ref{rem:detper} gives an example). While the complete classification of idempotents in a decorated incidence algebra is an interesting problem in its own right, their structure depends strongly on the decoration and a complete analysis appears is complicated. However, automorphisms of $\alg^\delta(G,\fie,\circ_\delta)$ induced by automorphisms of $G$ do preserve block idempotents, and the sets of block idempotents and their Peirce spectra are natural invariants attached to a decorated PSTS, which are the principal object of study in the present paper. Alternatively, as algebra automorphisms preserve the spectra of idempotents, the set of Peirce spectra of block idempotents is the same as the set of Peirce spectra of the algebra automorphism group orbit of the block idempotents, so is an invariant of the decorated incidence algebra.

For general decorated PSTSs, one expects the Peirce spectra of block idempotents to depend sensitively on the global pattern of signs
defining the decoration. However, Hsiang PSTSs exhibit a remarkable
rigidity phenomenon discovered in \cite{Fox-Tka2026b}. More precisely, to each Pasch
configuration one associates the gauge-invariant sign
$\kappa_\delta(P)\in\{\pm1\}$ given by the product of the signs of its blocks
(see \eqref{kappaP}). If
$$
\nu_\delta(b)
=
\#\{P\ni b:\kappa_\delta(P)=1\},
$$
denotes the number of positively signed Pasch configurations containing
a block $b$, then the Peirce spectrum of a block idempotent of a Hsiang PSTS is supported
on the fixed set
$$
\Lambda=
\left\{
1,-1,\tfrac12,-\tfrac12
\right\},
$$
independently of the decoration, while all spectral multiplicities are
completely determined by the single local parameter $\nu_\delta(b)$.
Thus, despite the freedom in choosing decorations, the entire Peirce spectrum is determined by the single integer $\nu_\delta(b)$.

\subsection{The main results}
Although this rigidity phenomenon is striking, its origin
remains unclear from a purely combinatorial point of view.
The known proof relies heavily on special structural properties of
Hsiang algebras, and it is not clear how these should be formulated
intrinsically in terms of the underlying PSTS.
This raises the question of whether this spectral rigidity is a peculiarity of Hsiang PSTSs or a feature shared by a broader class of decorated partial Steiner triple systems.

The present paper addresses this question for the combinatorial Grassmannian triple systems $G_2(m)$ (see \cite[Chapter $9.3$]{Brouwer-Cohen-Neumaier}, \cite{Prazmowska, Petelczyc-Prazmowska-Prazmowska, SanigaHolweckPracna}), defined as the PSTS on an $m$-element set $\Omega$ whose points are $2$-elements subsets of $\Omega$ and whose blocks are triples of points all contained in the same $3$-element subset of $\Omega$. These geometries possess a
rich and highly regular incidence structure, yet are considerably more
accessible combinatorially than the Hsiang systems. We show that a
closely related rigidity phenomenon persists in this setting and admits
a direct geometric interpretation in terms of Pasch
configurations. In this sense, the Grassmannian family provides a
natural testing ground for understanding the mechanism responsible for
spectral rigidity. However, both the
methods and the proofs developed here are independent of those
in \cite{Fox-Tka2026b}; the approach relies only on the intrinsic
combinatorics and incidence geometry of Grassmannian systems and makes no reference to
structural properties of Hsiang algebras.

The points, blocks, and Pasch configurations of the Grassmannian geometry $G_2(m)$ on the $m$-set $\Omega$ are parameterized by the $k$-subsets of $\Omega$, $\binom{\Omega}{k}$, for $k \in \{2, 3, 4\}$.
For $m = 5$ and $m = 6$ these yield the Desargues configuration
$G_2(5)=(10_3)$ and the Cayley--Salmon configuration $G_2(6)=(15_4,20_3)$ \cite{Petelczyc-Prazmowska-Prazmowska}.

A useful feature of the Grassmannian family is its recursive
structure. Fixing a block decomposes the geometry into several local
triangular pieces and a residual geometry isomorphic with
$G_2(m-3)$. The Peirce decomposition of the corresponding block idempotent mirrors
this decomposition: the local triangles account for all nonzero
eigenvalues, while the residual Grassmannian corresponds with the
zero eigenspace. Consequently, the associated copy of $G_2(m-3)$ inside
$G_2(m)$ is detected spectrally through the multiplicity of the zero
eigenvalue. This appears to be a distinctive
feature of the Grassmannian family that has no direct analogue in the
general PSTS setting considered in \cite{Fox-Tka2026b}.

The first of our main results, stated precisely in Theorem~\ref{thm:localpaschspectrum}, shows that the spectrum of a block idempotent is controlled by a simple combinatorial invariant. Given a block $b$, a point of $\Omega\setminus\supp(b)$ determines a unique Pasch configuration $P$ containing $b$ and hence contributes a sign $\kappa_\delta(P) = \pm1$; let
$$
n_{-1}^{\delta}(b)
$$
denote the number of occurrences of the sign $-1$. We prove that the complete multiset spectrum of the corresponding block idempotent is determined by the ambient parameter $m$ and the integer $n_{-1}^{\delta}(b)$, all spectral multiplicities being explicit functions of these two quantities. In particular, the multiplicity of the Peirce eigenvalue $-1$ is precisely the number of negatively signed Pasch configurations containing the chosen block. Moreover, every value
$$
0\le n_{-1}^{\delta}\le m-3
$$
is realizable by a suitable decoration, so that $G_2(m)$ carries exactly $m-2$ distinct local spectral types; for example, $G_2(5)$ admits three types and $G_2(6)$ admits four.

It is instructive to compare this with the situation for Hsiang PSTSs, where the construction singles out two distinguished decorations: for the incidence decoration all Pasch configurations are positive, while for the Hsiang decoration all are negative. In both cases every block idempotent has the same multiset spectrum, so that the corresponding spectral profile is supported on a single spectral type. These two spectrally homogeneous realizations of the same geometry retain a vestige of the strong rigidity of Hsiang algebras, for which the spectrum is independent of the choice of block idempotent. In the Grassmannian family this fails. The incidence decoration still yields a single spectral type, but for $G_2(m)$ most other decorations give rise to several distinct spectral types simultaneously.

The local classification leads naturally to a global one. The \emph{spectral profile} of a decorated Grassmannian geometry records how many blocks carry each of the local spectral types. It is identified with the degree distribution of the $4$-uniform hypergraph whose hyperedges are the negative Pasch configurations. We then develop a hierarchy of higher-moment invariants, interpret them as overlap counts for families of negative Pasch configurations sharing a common block, and show that the spectral profile is uniquely recoverable from these invariants by binomial inversion. Since Pasch signs are unaffected by gauge transformations, all these invariants depend only on the gauge-equivalence class of the decoration. Taken together, the results reveal a direct connection between Grassmannian incidence geometry, Pasch combinatorics, hypergraph theory, and the spectral theory of commutative nonassociative algebras.

Two further results place the Grassmannian family within a wider algebraic
context. First, both the plain incidence algebra of $G_2(m)$ and its decorated incidence algebra for the star-decorations of Definition~\ref{def:stardec} admit
concrete matrix realizations. Given $\omega \in \Omega$, the star decoration $\delta_{\omega}$ assigns $-1$ to each block $b$ whose support contains $\omega$, so that for every such block, the number $n_{-1}^{\delta}(b)$ of negative Pasch configurations containing the block is the maximum possible, $m-3$.
The incidence algebra of $G_2(m)$ is
isomorphic to the space of symmetric $m\times m$ matrices with vanishing diagonal
equipped with the projected Jordan product.
Adjoining to this algebra a copy of the underlying Euclidean space $\ualg$, with a product
combining the Jordan structure on matrices with the ordinary
matrix-vector action, produces a second explicit algebra
$\vm{\ualg}=\ualg\oplus\Sym_0(\ualg)$, and we show that this
is isomorphic to the star-decorated algebra
$\alg^{\delta_\omega}(G_2(m))$ for $\ualg$ of dimension $m-1$.
Under this
identification the block idempotents at $\omega$ correspond to simple
rank-one data built from unit vectors of $\ualg$, the algebra is
generated by $m-1$ of them, and $\alg^{\delta_\omega}(G_2(m))$
decomposes as $\ualg\oplus\alg^{\mathbf1}(G_2(m-1))$, the natural module
of $\ualg$ together with the undecorated incidence algebra one dimension down, $\alg^{\mathbf1}(G_2(m-1))$.
Corollary~\ref{cor:gradingauto} shows this decomposition is a canonical $\mathbb Z_2$-grading. The two smallest members of this matrix-model family turn out to be exact
radial Hsiang algebras, more precisely, polar algebras built from a classical symmetric Clifford system \cite{Ferus-Karcher-Munzner, Tkachev-cliff}; see Examples~\ref{ex:polar3} and \ref{ex:mutant}.

Second, the star decorations turn out to be distinguished also from the point of view of axial algebras \cite{Khasraw-McInroy-Shpectorov}. For any decoration, a block idempotent is an axis precisely when every Pasch configuration through the block is negative, equivalently when the local spectral parameter attains its maximal value, in which case the Peirce decomposition refines the canonical $\mathbb Z_2$-grading of
Corollary~\ref{cor:gradingauto}. For a general decoration it could occur that a block idempotent is an axis but the decorated incidence algebra is not axial.
For the star decoration $\delta_\omega$, every block through $\omega$ is an axis and no other block is; since these idempotents generate the algebra (Theorem~\ref{thm:vm}(c)), $\alg^{\delta_\omega}(G_2(m))$ is an axial algebra whose axis set is the full star of blocks at $\omega$. All these axes obey the same fusion law, a sign-doubled analogue of the classical Jordan-type law $\{1,0,\tfrac12\}$ on the enlarged eigenvalue set
$\{1,-1,\tfrac12,-\tfrac12,0\}$, independent of $m$.

Finally, a word of caution is in order with regards to extensions of these results. The distinguished role played by Pasch configurations in the present theory is not a generic feature of partial Steiner triple systems. Rather, it reflects a special property of the Hsiang and Grassmannian geometries. In general, the spectral analysis of a block idempotent is governed by the local incidence geometry of the block.
For these two families the local incidence geometry is completely encoded by Pasch configurations, which are therefore the natural carriers of the local spectral information. In other classes of Steiner systems, such as the Fano plane or the projective geometries $PG(n,2)$, the corresponding local incidence structure is controlled by specific configurations more complicated than Pasch configurations; consequently, Pasch geometry alone is generally insufficient to describe the full spectral behavior of the associated incidence algebras. A systematic study will be pursued elsewhere.

\medskip
The paper is organized as follows. Section~\ref{sec:not} fixes notations and records basic facts on Grassmannian geometries and their decorated incidence algebras. Section~\ref{sec:proof} is devoted to the spectral analysis of block idempotents: we establish the block decomposition of $G_2(m)$ induced by a fixed block, compute the associated Peirce decomposition, and prove Theorem~\ref{thm:localpaschspectrum}. Section~\ref{sec:matrixmodel} realizes the incidence algebra of $G_2(m)$, and the star-decorated algebras of $G_2(m)$, as concrete algebras of symmetric matrices, and records the resulting second self-similarity $\alg^{\delta_\omega}(G_2(m))\cong\ualg\oplus\alg^{\mathbf1}(G_2(m-1))$. Section~\ref{sec:axial} shows that block idempotents are axes precisely for the maximally negative blocks, that the star decorations realize this on the full star of blocks at a point, and computes the resulting fusion law, exhibiting $\alg^{\delta_\omega}(G_2(m))$ as an axial algebra. In Section~\ref{sec:dist} we turn to the global distribution of local spectral types, introduce the spectral profile, identify it with the degree distribution of the negative-Pasch hypergraph, and derive the hierarchy of higher-moment and overlap invariants together with the corresponding binomial inversion formula. Section~\ref{sec:exa} contains explicit computations for the smallest Grassmannian geometries. The last section discusses gauge equivalence of decorations, presents enumerative results on realizable spectral profiles, and formulates several open questions.

\section{Spectral analysis of block idempotents}

\subsection{Preliminaries}\label{sec:not}

For a finite set $X$, $\binom{X}{k}$ denotes the set of $k$-element subsets of $X$. Fix an integer $m\ge2$. Throughout the paper $\Omega:=\{1,\dots,m\}$.
The Grassmannian geometry $G_2(m)$ is the partial Steiner triple system whose point set is
$\mathcal P=\binom{\Omega}{2}$,
and whose set of blocks $\mathcal B$ comprises the triples of pairwise collinear points
$b_{ijk}:=\{ij,ik,jk\}$, which are parametrized by $3$-element subset $\{i,j,k\}\in\binom{\Omega}{3}$. Their cardinalities are $|\mathcal P|=\binom{m}{2}$ and $|\mathcal B|=\binom{m}{3}$.
Similarly, quadrangles (Pasch configurations) of $G_2(m)$ are parametrized by $\binom{\Omega}{4}$. We shall routinely identify the point $\{i,j\}\in\binom{\Omega}{2}$ with the symbol $ij$.

The \emph{support} of the block $b\in\mathcal B$,
\begin{equation}\label{eq:suppdef}
\supp(b):=\bigcup_{p\in b}p\ \subset\ \Omega,
\end{equation}
is the triple of the elements of $\Omega$ contained in the points of $b$, so that $\supp(b_{ijk}) = \{i,j,k\}$. The map $b\mapsto\supp(b)$ is a bijection between $\mathcal B$ and $\binom{\Omega}{3}$, with inverse $\{i,j,k\}\mapsto b_{ijk} = \{ij,ik,jk\}$.

Next we describe the decorated incidence algebras associated with the combinatorial geometry $G_2(m)$. Given a \textit{decoration} $\delta:\mathcal B\to\{\pm1\}$,
the \textit{incidence algebra} $\alg^\delta(G_2(m))$ is the vector space
$\valg=\bigoplus_{p\in\mathcal P}\fie e_p$ with basis
$\{e_p:p\in\mathcal P\}$ equipped with the commutative multiplication given by extending bilinearly the products
\begin{equation}\label{multiplication}
e_{ij}\circ_\delta e_{ik}
=
\delta(\{ij,ik,jk\})\,e_{jk},
\qquad
\{i,j,k\}\in\binom{\Omega}{3},
\end{equation}
together with
\begin{align}\label{multiplication2}
&e_p\circ_\delta e_q=0,& &p, q \in \mathcal P \text{ not collinear.}
\end{align}
In particular, $e_p\circ_\delta e_p = 0$ so the generators $e_p$ are $2$-nilpotents.


The following definition from \cite{Fox-Tka2026b} plays a central role in the spectral theory of decorated partial Steiner triple systems.

\begin{definition}[\cite{Fox-Tka2026b}]\label{def:paschsign}
Let $\Pi=(\mathcal P,\mathcal B)$ be a decorated partial Steiner triple system with decoration
$\delta:\mathcal B\to\{\pm1\}.$
If $P=\{b_1,b_2,b_3,b_4\}\subseteq\mathcal B$
is a Pasch configuration, then its \emph{$\delta$-sign} is defined by
\begin{equation}\label{kappaP}
\kappa_\delta(P)
=
\prod_{b\in P}\delta(b)
=
\delta(b_1)\delta(b_2)\delta(b_3)\delta(b_4).
\end{equation}
\end{definition}

In the Grassmannian geometry $G_2(m)$, a Pasch configuration is uniquely determined by a $4$-subset
$\rho = \{i,j,k,r\}\in\binom{\Omega}{4}$. Precisely, $\rho$ determines the Pasch configuration
$$
P_\rho
=
\{
b_{ijk},
b_{ijr},
b_{ikr},
b_{jkr}
\}.
$$
Accordingly, it is convenient to write
\begin{equation}\label{kappaijkr}
\kappa_{ijkr}
:=
\kappa_\delta(P_\rho)
=
\delta(b_{ijk})
\delta(b_{ijr})
\delta(b_{ikr})
\delta(b_{jkr}).
\end{equation}

\subsection{A block decomposition of $G_2(m)$}\label{sec:proof}
Fix a block $b\in\mathcal B$, write $\{i,j,k\}:=\supp(b)$ so that $b=\{ij,ik,jk\}$, and define the cardinality $m-3$ subset $R:=\Omega\setminus\{i,j,k\}$. The underlying point set of $G_2(m)$ decomposes naturally as disjoint union
\begin{equation}\label{Pdecomp}
\mathcal{P}
=
\mathcal{P}_0
\sqcup
\mathcal{P}_1
\sqcup
\mathcal{P}_2,
\end{equation}
of the subsets
\begin{align*}
&\mathcal{P}_0=
b
=
\{ij,ik,jk\},&
&\mathcal{P}_1=
\bigsqcup_{r\in R}
\{ir,jr,kr\},
&&
\mathcal{P}_2=
\{rs:r,s\in R,\ r\ne s\},
\end{align*}
having cardinalities
\begin{align*}
&|\mathcal{P}_0|=3,&&
|\mathcal{P}_1|=3(m-3),&&
|\mathcal{P}_2|=\binom{m-3}{2}.
\end{align*}
For each $r\in R$, define
$$
W_r:=\{ir,jr,kr\},
\qquad
r\in R.
$$
As $\mathcal{P}_2 =
\binom{R}{2}$ consists of all points completely outside the fixed block, it is identified with the point set of the Grassmannian
$G_2(m-3)$ based on the $(m-3)$-set $R$. Thus the decomposition \eqref{Pdecomp} can be rewritten slightly imprecisely as
$$
\mathcal{P}
=
b
\sqcup
\Bigl(
\bigsqcup_{r\in R}W_r
\Bigr)
\sqcup
G_2(m-3).
$$
Passing from geometry to algebra there results a corresponding decomposition
\begin{equation}\label{Vdecomp}
\valg
=
\valg_{\mathcal{P}_0}
\oplus
\Bigl(
\bigoplus_{r\in R}\valg_{W_r}
\Bigr)
\oplus
\valg_{\mathcal{P}_2},
\end{equation}
where
$\valg_S
:=
\fie\{ e_p:\ p\in S\}$
denotes the linear span of $S$.

\subsection{A cohomological complex and a canonical $\mathbb Z_2$-grading}\label{sec:d-complex}

Both gauge transformations and Pasch-sign patterns admit a natural cohomological
interpretation. For Grassmannian geometries the relevant simplicial
complex is the full simplex on the vertex set $\Omega$: its
$1$-simplices are the points of $G_2(m)$, its $2$-simplices are the
blocks, and its $3$-simplices are the Pasch configurations.

Identifying $\{\pm1\}$ with $\mathbb F_2$, vertex signs
$\eta:\Omega\to\{\pm1\}$, point sign changes $\varepsilon$,
decorations $\delta$, and Pasch-sign patterns $\kappa_\delta$
can be regarded as $0$-, $1$-, $2$-, and $3$-cochains, respectively. The corresponding
coboundary operators
\begin{equation}\label{ddef2}
C^0\stackrel{d_0}{\longrightarrow}
C^1\stackrel{d_1}{\longrightarrow}
C^2\stackrel{d_2}{\longrightarrow}
C^3 .
\end{equation}
are given by
\begin{align}\label{d0d1d2}
\begin{aligned}
(d_0\eta)(ij)
&=
\eta(i)\eta(j),\\
(d_1\varepsilon)(ijk)
&=
\varepsilon(ij)\varepsilon(jk)\varepsilon(ki),\\
(d_2\delta)(ijkl)
&=
\delta(ijk)\delta(ijl)\delta(ikl)\delta(jkl),
\end{aligned}
\end{align}
and coincide with the assignments
\begin{equation}\label{ddef1}
\eta\stackrel{d_0}{\longmapsto}\varepsilon,
\qquad
\varepsilon\stackrel{d_1}{\longmapsto}\delta^\varepsilon/\delta,
\qquad
\delta\stackrel{d_2}{\longmapsto}\kappa_\delta.
\end{equation}
Indeed, the three points of the block $\{i,j,k\}\in\binom{\Omega}{3}$ are
$ij$, $jk$ and $ki$, so that
$$
\delta^\varepsilon(\{i,j,k\})
=
\varepsilon(ij)\varepsilon(jk)\varepsilon(ki)\,
\delta(\{i,j,k\})
=
(d_1\varepsilon)(\{i,j,k\})\,\delta(\{i,j,k\}),
$$
that is, the gauge action is translation of the decoration space by the
subgroup $\operatorname{im}d_1$, while the Pasch-sign pattern is obtained
from the decoration by the next coboundary operator,
$\kappa_\delta=d_2\delta$. In particular, the gauge invariance of the
Pasch signs recorded above is nothing but the identity
\begin{equation}\label{d2d1eq}
d_2\circ d_1=0.
\end{equation}
Since the full simplex is contractible, this complex is exact in positive
degrees. Consequently,
\begin{equation}\label{rankd2}
\ker d_2=\operatorname{im}d_1,
\end{equation}
so that two decorations of $G_2(m)$ have the same Pasch-sign pattern if
and only if they are gauge equivalent; in particular, the realizable
Pasch-sign patterns are in bijection with the gauge-equivalence classes.
Furthermore,
$$
\ker d_1=\operatorname{im}d_0,
$$
so the gauge kernel consists exactly of the sign changes
$$
\varepsilon(ij)=\eta(i)\eta(j),
\qquad
i,j\in\Omega,
$$
and is therefore isomorphic to $\mathbb F_2^{m-1}$. It follows that
\begin{equation}\label{rankd1}
\operatorname{rank}d_1
=
\binom{m}{2}-(m-1)
=
\binom{m-1}{2}.
\end{equation}

\begin{lemma}\label{lem:gradingauto}
For a subset $\tau\subseteq\Omega$ define
$\varepsilon_\tau:\mathcal P\to\{\pm1\}$ by
\begin{equation}\label{epstau}
\varepsilon_\tau(p)=(-1)^{|p\cap\tau|},
\qquad
p\in\mathcal P,
\end{equation}
that is, $\varepsilon_\tau=d_0(\mathbf 1_\tau)$, where
$\mathbf 1_\tau\in C^0$ is the indicator function of $\tau$. Then, for
\emph{every} decoration $\delta$, the diagonal map
\begin{equation}\label{thetatau}
\theta_\tau:\ \alg^\delta(G_2(m))\to\alg^\delta(G_2(m)),
\qquad
\theta_\tau(e_p)=\varepsilon_\tau(p)\,e_p ,
\end{equation}
is an involutive automorphism, and the induced decomposition
\begin{equation}\label{taugrading}
\begin{aligned}
\alg^\delta&=\alg_{+}(\tau)\oplus\alg_{-}(\tau),\\
\alg_{\pm}(\tau)&=\valg_{\mathcal P^{\pm}(\tau)},\\
\mathcal P^{\pm}(\tau)&=\{p\in\mathcal P:\ (-1)^{|p\cap\tau|}=\pm1\},
\end{aligned}
\end{equation}
is a $\mathbb Z_2$-grading of $\alg^\delta$.
\end{lemma}

\begin{proof}
Clearly $\theta_\tau^2=\mathrm{id}$. That $\theta_{\tau}$ is a homomorphism need only be tested
on basis vectors. If $p,q$ are not collinear then
$e_p\circ_\delta e_q=0$, so
$$\theta_\tau(e_p)\circ_\delta\theta_\tau(e_q)
= \epsilon_{\tau}(p)\epsilon_{\tau}(q)e_{p}\circ e_{q} = 0 = \theta_{\tau}(e_{p}\circ_{\delta}e_{q}).$$
If $p,q,r$ are the three points of a block $b'$,
then $e_p\circ_\delta e_q=\delta(b')e_r$, so that
$$
\theta_\tau(e_p)\circ_\delta\theta_\tau(e_q)
=\varepsilon_\tau(p)\varepsilon_\tau(q)\,\delta(b')e_r,
\qquad
\theta_\tau(e_p\circ_\delta e_q)=\varepsilon_\tau(r)\,\delta(b')e_r .
$$
These agree for every block $b'$ if and only if
$\varepsilon_\tau(p)\varepsilon_\tau(q)\varepsilon_\tau(r)=1$, that is,
if and only if $d_1\varepsilon_\tau\equiv1$, which holds because
$d_1\varepsilon_\tau=d_1d_0(\mathbf 1_\tau)=0$. This proves that
$\theta_\tau$ is an automorphism, and \eqref{taugrading} is its
eigenspace decomposition by \eqref{epstau}.
\end{proof}

Specializing $\tau=\supp{b}$ yields, for every block $b$, a canonical grading automorphism of the incidence
algebra independently of the decoration $\delta$.

\begin{corollary}[A canonical $\mathbb Z_2$-grading]
\label{cor:gradingauto}
Let $b\in\mathcal B$ and put
$\varepsilon_b:=\varepsilon_{\supp(b)}$, $\theta_b:=\theta_{\supp(b)}$,
so that
\begin{equation}\label{epsilonmap}
\varepsilon_b(p)=(-1)^{|p\,\cap\,\supp(b)|},
\qquad
\theta_b(e_p)=\varepsilon_b(p)\,e_p .
\end{equation}
Then, for every decoration $\delta$, the map $\theta_b$ is an involutive
automorphism of $\alg^\delta(G_2(m))$, and its eigenspace decomposition
is
\begin{align}\label{apriorigrading}
\begin{aligned}
&\alg^\delta=\alg_{+}(b)\oplus\alg_{-}(b),\\
&\alg_{+}(b)=\valg_{\mathcal P_0}\oplus\valg_{\mathcal P_2},&&
&\alg_{-}(b)=\valg_{\mathcal P_1}=\bigoplus_{r\in R}\valg_{W_r} .
\end{aligned}
\end{align}
\end{corollary}

\begin{proof}
Apply Lemma~\ref{lem:gradingauto} with $\tau=\supp(b)$. By
\eqref{Pdecomp}, $|p\cap\supp(b)|$ equals $2$ for $p\in\mathcal P_0$,
$1$ for $p\in\mathcal P_1$ and $0$ for $p\in\mathcal P_2$; hence
$\mathcal P^{+}(\tau)=\mathcal P_0\sqcup\mathcal P_2$ and
$\mathcal P^{-}(\tau)=\mathcal P_1$, which is \eqref{apriorigrading}.
\end{proof}

\subsection{The Peirce decomposition}
The decomposition \eqref{Vdecomp} is  a consequence of the fact that the multiplication in $\alg^\delta(G_2(m))$ is compatible with the natural inclusions of smaller Grassmannian geometries. The following lemma makes this observation precise.

\begin{lemma}\label{lem:closure}
For $T\subseteq\Omega=\{1,\dots,m\}$ the subspace
$\valg(T):=\fie\{e_{ij}: i,j\in T\}$ is a subalgebra of $\alg^\delta(G_2(m))$ having dimension $
\dim \valg(T)=\binom{|T|}{2}$, and such that if $\delta_T$ denotes the restriction of $\delta$ to the blocks contained in $T$, then $\valg(T)$ is isomorphic with $\alg^{\delta_T}(G_2(|T|))$,
$\valg(T)\simeq\alg^{\delta_T}(G_2(|T|))$.
\end{lemma}

\begin{proof}
If $p,q\in\mathcal P(T)$ are not collinear, then $e_p\circ_\delta e_q=0$. If $p=ij$ and $q=ik$, then \[ e_{ij}\circ_\delta e_{ik} = \delta(b_{ijk})e_{jk}, \] and since $j,k\in T$, the product again belongs to $\valg(T)$. Hence $\valg(T)$ is a subalgebra. The remaining assertions are immediate.
\end{proof}

\begin{corollary}\label{cor:where} Let $S\subseteq\mathcal{P}$ and let $$ T:=\bigcup_{ij\in S}\{i,j\}. $$ Then the subalgebra generated by $\{e_p:p\in S\}$ is contained in $\valg(T).$ \end{corollary}

The decomposition \eqref{Vdecomp} is adapted to the geometry of the fixed block and serves as the starting point for the spectral analysis. Proposition \ref{prop:blockdecomp} shows that it is preserved by multiplication with the block idempotent.

\begin{proposition}\label{prop:blockdecomp}
For any decoration $\delta:\mathcal{B}\to\{\pm1\}$ and any block $b\in\mathcal B$, the decomposition \eqref{Vdecomp} is invariant under $L_\delta(c_b^\delta)$. Moreover, with respect to \eqref{Vdecomp}, the operator $L_\delta(c_b^\delta)$ is block diagonal,
\begin{equation}\label{blockdiag}
L_\delta(c_b^\delta)
=
L_0
\oplus
\Bigl(
\bigoplus_{r\in R}L_r
\Bigr)
\oplus
0_{\mathcal{P}_2}.
\end{equation}
Here $L_0$ is the restriction of $L_\delta(c_b^\delta)$ to $\valg_{\mathcal{P}_0}$, $L_r$ is its restriction to $\valg_{W_r}$, and $0_{\mathcal{P}_2}$ denotes the zero operator on $\valg_{\mathcal{P}_2}$.
\end{proposition}

\begin{proof}
Write $c=c_b^\delta.$ Since $c$ is supported on
$$
\mathcal{P}_0=b=\{ij,ik,jk\},\quad \{i,j,k\}=\supp(b),
$$
it suffices to determine the action of $c$ on the three summands of \eqref{Vdecomp}.
Since $\mathcal{P}_0$ is a block, by Corollary~\ref{cor:where}, $\valg_{\mathcal{P}_0}$ is a subalgebra containing $c$, so
\begin{equation}\label{B0inv}
L_\delta(c)(\valg_{\mathcal{P}_0})
\subseteq
\valg_{\mathcal{P}_0}.
\end{equation}
For $W_r=\{ir,jr,kr\}$, $r\in R$, there hold
\begin{align}\label{wr3}
&e_{ir}\circ_\delta e_{ij}=\delta(b_{ijr}) e_{jr},&
&e_{ir}\circ_\delta e_{ik}=\delta(b_{ikr}) e_{kr},&
&e_{ir}\circ_\delta e_{jk}=0,
\end{align}
and similar identities hold after cyclic permutation of $i,j,k$. Therefore
\begin{equation}\label{wrinv}
L_\delta(c)(\valg_{W_r})
\subseteq
\valg_{W_r}
\end{equation}
for every $r\in R$. Finally, if $rs\in\mathcal{P}_2$, then $rs$ does not belong to any block containing a point of $\mathcal{P}_0$, so
\begin{align}\label{B2a}
&e_{rs}\circ_\delta e_{ij}=0,
&&
e_{rs}\circ_\delta e_{ik}=0,
&&
e_{rs}\circ_\delta e_{jk}=0,
\end{align}
and consequently
\begin{equation}\label{B2zero}
L_\delta(c)(\valg_{\mathcal{P}_2})=0.
\end{equation}
Combining \eqref{B0inv}, \eqref{wrinv}, and \eqref{B2zero} yields the block decomposition \eqref{blockdiag}.
\end{proof}

The block decomposition reduces the spectral problem to the computation of the individual diagonal blocks. We now determine these blocks explicitly.

\begin{proposition}\label{prop:blockmatrices}
Let $b\in\mathcal B$, write $\{i,j,k\}:=\supp(b)$, and let $c=c_b^\delta$. The matrix of $L_0$ with respect to the basis $\{e_{ij},e_{ik},e_{jk}\}$ is
\begin{equation}\label{L0matrix}
L_0
=
\tfrac12
\begin{pmatrix}
0&1&1\\
1&0&1\\
1&1&0
\end{pmatrix},
\end{equation}
and therefore
\begin{equation}\label{specL0}
\Spec(L_0)
=
\Bigl\{
1,
\Bigl(-\tfrac12\Bigr)^2
\Bigr\}.
\end{equation}
Furthermore, for every $r\in R$, the characteristic polynomial of $L_r$ is
\begin{equation}\label{charpolyr}
\chi_r(\lambda)
=
\lambda^3-\tfrac34\lambda-\tfrac14\kappa_{ijkr}.
\end{equation}
Consequently,
\begin{equation}\label{specplus}
\Spec(L_r)
=
\left\{
  \begin{array}{ll}
    \bigl\{1,\bigl(-\tfrac12\bigr)^2\bigr\}, & \hbox{if $\kappa_{ijkr}=1$;} \\
    \bigl\{-1,\bigl(\tfrac12\bigr)^2\bigr\}, & \hbox{if $\kappa_{ijkr}=-1$.}
  \end{array}
\right.
\end{equation}
In particular,
\begin{equation}\label{fiveeigen}
\spec\bigl(L_\delta(c_b^\delta)\bigr)
\subseteq
\Bigl\{
0,\pm1,\pm\tfrac12
\Bigr\}.
\end{equation}
\end{proposition}

\begin{proof}
By Proposition \ref{prop:blockdecomp}, the operator $L_\delta(c)$ preserves each summand in \eqref{Vdecomp}, acts on $\valg_{\mathcal{P}_2}$ as zero, and decomposes as in \eqref{blockdiag}. It remains only to compute the non-zero blocks. We compute separately the blocks $L_0$ and $L_r$.
First consider the summand $\valg_{\mathcal{P}_0}$, where $\mathcal{P}_0=\{ij,ik,jk\}$.
By definition of the block idempotent $c_b^\delta$, using $e_{ij}\circ_\delta e_{ij}=0$ and $\delta(b)^2=1$, a direct computation gives
$$
L_\delta(c)e_{ij}
=
\tfrac12\delta(b)\bigl(e_{ij}\circ_\delta e_{ij}+e_{ik}\circ_\delta e_{ij}+e_{jk}\circ_\delta e_{ij}\bigr)
=
\tfrac12\bigl(e_{jk}+e_{ik}\bigr),
$$
and similarly for $e_{ik}$ and $e_{jk}$. Hence the restriction of $L_\delta(c)$ to $\valg_{\mathcal{P}_0}$ is represented by the matrix $L_{0}$ in \eqref{L0matrix} which has characteristic polynomial
$$
\chi_0(\lambda)
=
\lambda^3-\tfrac34\lambda-\tfrac14 =
(\lambda-1)\Bigl(\lambda+\tfrac12\Bigr)^2.
$$
This proves \eqref{specL0}.
Now fix $r\in R$ and consider $W_r=\{ir,jr,kr\}$. The corresponding Pasch configuration $
P:=\{b_{ijk},
b_{ijr},
b_{ikr},
b_{jkr}\}
$ has $\delta$-sign $\kappa_{ijkr}$ as in \eqref{kappaijkr}.
In the ordered basis $(e_{ir},e_{jr},e_{kr}),$
the operator $L_r$ is represented by a matrix of the form
$$
L_r
=
\frac12
\begin{pmatrix}
0&s_1&s_2\\
s_1&0&s_3\\
s_2&s_3&0
\end{pmatrix},
$$
where
\begin{align}
s_1&=\delta(b_{ijk})\delta(b_{ijr}),
\quad
s_2=\delta(b_{ijk})\delta(b_{ikr}),
\quad
s_3=\delta(b_{ijk})\delta(b_{jkr}).
\label{eps3}
\end{align}
Since $\delta(b_{ijk})^2=1$, it follows that
\begin{equation}\label{epsproduct}
s_1s_2s_3
=
\delta(b_{ijk})
\delta(b_{ijr})
\delta(b_{ikr})
\delta(b_{jkr})
=
\kappa_{ijkr}.
\end{equation}

Using \eqref{epsproduct}, a  direct determinant calculation yields
$$
\begin{aligned}
\chi_r(\lambda)&=\det(\lambda I-L_r) =\lambda^3-\tfrac34\lambda-\tfrac14 s_1s_2s_3
\\&= \lambda^3-\tfrac34\lambda-\tfrac14\kappa_{ijkr}  = (\lambda-\kappa_{ijkr})\bigl(\lambda+\tfrac12 \kappa_{ijkr}\bigr)^2,
\end{aligned}
$$
which shows \eqref{charpolyr} and \eqref{specplus}.
Finally, by \eqref{blockdiag}, the remaining block is $0_{\mathcal{P}_2}$. Combining \eqref{specL0}, \eqref{specplus} and \eqref{B2zero} yields \eqref{fiveeigen}.
\end{proof}

\subsection{Proof of the main result}\label{sec:localpasch}

Combining Propositions~\ref{prop:blockdecomp} and \ref{prop:blockmatrices}, we obtain the main result of the paper, which expresses the complete multiset spectrum of a block idempotent in terms of the number of negative local Pasch configurations containing the corresponding block.

\begin{theorem}\label{thm:localpaschspectrum}
Let $b\in\mathcal B$, write $\{i,j,k\}:=\supp(b)$, let $R:=\Omega\setminus\{i,j,k\}$, and let
\begin{equation}\label{kjkr}
\kappa_{ijkr}
=
\delta(b)
\delta(b_{ijr})
\delta(b_{ikr})
\delta(b_{jkr}),
\qquad
r\in R.
\end{equation}
Define
\begin{equation}\label{nminusone1}
n^{\delta}_{-1}(b)
=
|\{r\in R:\kappa_{ijkr}=-1\}|,
\end{equation}
which we abbreviate to $n^{\delta}_{-1}$ whenever the block $b$ is
fixed. Then
\begin{equation}\label{nplusrelation}
|\{r\in R:\kappa_{ijkr}=1\}| = m-3-n^{\delta}_{-1}
\end{equation}
and
\begin{equation}\label{localspectrumformula}
\begin{aligned}
\Spec&(L_\delta(c_b^\delta))=\sigma_{n^{\delta}_{-1}}
\end{aligned}
\end{equation}
where
\begin{equation}\label{sigmanot}
\sigma_k
=
\Bigl\{
0^{\binom{m-3}{2}},
1^{m-2-k},
(-1)^k,
(-\tfrac12)^{2(m-2-k)},
(\tfrac12)^{2k}
\Bigr\}
\end{equation}
\end{theorem}

\begin{proof}
Put $c:=c_b^\delta.$
Since $|R|=m-3,$
it follows from the definition of $n^{\delta}_{-1}$ that
\begin{equation}\label{partitionR}
R=
\{r\in R:\kappa_{ijkr}=1\}
\sqcup
\{r\in R:\kappa_{ijkr}=-1\}.
\end{equation}
Taking cardinalities in \eqref{partitionR} yields
\begin{equation}\label{nplusrelationproof}
|\{r\in R:\kappa_{ijkr}=1\}|
=
m-3-n^{\delta}_{-1},
\end{equation}
which proves \eqref{nplusrelation}.
By Proposition \ref{prop:blockdecomp}, $L_\delta(c)$ preserves the decomposition \eqref{Vdecomp} with respect to which it has the block form \eqref{blockdiag}.
Therefore
$\Spec(L_\delta(c))$
is obtained by combining the spectra of the diagonal blocks. By Proposition \ref{prop:blockmatrices}, the spectra of the blocks $L_0$ and $L_r$ are given in \eqref{specL0} and \eqref{specplus}.
By \eqref{nplusrelationproof}, there are $m-3-n^{\delta}_{-1}$ elements $r\in R$ satisfying $\kappa_{ijkr}=1$, and $n^{\delta}_{-1}$ elements $r\in R$ satisfying $\kappa_{ijkr}=-1$.
Hence the blocks $L_r$ contribute
\begin{equation}\label{Wrcontribution}
\Bigl\{
1^{m-3-n^{\delta}_{-1}},
(-1)^{n^{\delta}_{-1}},
(-\tfrac12)^{2(m-3-n^{\delta}_{-1})},
(\tfrac12)^{2n^{\delta}_{-1}}
\Bigr\}
\end{equation}
to the multiset spectrum.
The only remaining contribution comes from the residual Grassmannian component. Proposition \ref{prop:blockdecomp} implies that $L_\delta(c)$ vanishes on $\valg_{\mathcal{P}_2}$,
and since
$
\dim\valg_{\mathcal{P}_2}
=
|\mathcal{P}_2|
=
\binom{m-3}{2},
$
this block contributes $0^{\binom{m-3}{2}}.$ Combining this with \eqref{specL0} and \eqref{Wrcontribution} yields \eqref{localspectrumformula}.
\end{proof}

An immediate consequence of Theorem~\ref{thm:localpaschspectrum} is that the spectral classification of block idempotents reduces to the realization problem for the single parameter $n_{-1}^\delta$.

\begin{corollary}\label{cor:realizable}
The spectrum of a block idempotent in
$\alg^\delta(G_2(m))$
is completely determined by the single integer
$$
n^{\delta}_{-1}
=
|\{r\in R:\kappa_{ijkr}=-1\}|.
$$
Moreover, for every decoration of $G_2(m)$,
$$
0\le n^{\delta}_{-1}\le m-3.
$$
Conversely, for every block $b\in\mathcal B$ and every integer
$
0\le n\le m-3,
$
there exists a decoration $\delta$ of $G_2(m)$ for which
\begin{equation}\label{realizablevalue}
n^{\delta}_{-1}(b)=n.
\end{equation}
\end{corollary}

\begin{proof}
The first assertion is immediate from Theorem
\ref{thm:localpaschspectrum}. Since
$R=\Omega\setminus\{i,j,k\}
$ satisfies $|R|=m-3$, the identity
$$
|\{r\in R:\kappa_{ijkr}=1\}|
=
m-3-n^{\delta}_{-1}
$$
implies
$0\le n^{\delta}_{-1}\le m-3.$
In the converse direction, to prove realizability, fix
$0\le n\le m-3$. Choose a cardinality $n$ subset $S\subseteq R$ and define a decoration $\delta$ as follows. Set
$\delta(b)=1$ and, for every $r\in R$, define
$$
\delta(b_{ijr})
=
\left\{
\begin{array}{ll}
-1,
&
r\in S,
\\
1,
&
r\notin S.
\end{array}
\right.
$$
Finally, for every remaining block $b$, including all blocks of the form
$b_{ikr}$ and $b_{jkr}$, $r\in R$, set
$\delta(b)=1$.
By construction,
$$
\kappa_{ijkr}
=
\delta(b)
\delta(b_{ijr})
\delta(b_{ikr})
\delta(b_{jkr})
=
\delta(b_{ijr}),
$$
so that
$$
\kappa_{ijkr}=-1
\quad\Longleftrightarrow\quad
r\in S.
$$
Hence
$$
n^{\delta}_{-1}
=
|\{r\in R:\kappa_{ijkr}=-1\}|
=
|S|
=
n.
$$
Thus, for the given block $b$, there exists a decoration of
$G_2(m)$ such that \eqref{realizablevalue} holds.
\end{proof}

\begin{definition}
For each $0\le k\le m-3$, the multiset
$\sigma_k$ defined by \eqref{sigmanot} is called the $k$-th \textit{spectral type}.
\end{definition}

Together Theorem~\ref{thm:localpaschspectrum}
and Corollary~\ref{cor:realizable} show that
$
\Spec(L_\delta(c_b^\delta))
=
\sigma_{n^{\delta}_{-1}(b)},
$
and that all $m-2$ spectral types $\sigma_0,\dots,\sigma_{m-3}$ occur.

Among all decorations, one family is particularly natural from both the
combinatorial and the cohomological point of view. These decorations are
attached to individual vertices of $\Omega$ and will play
a central role below: they admit the concrete matrix realizations of
Section~\ref{sec:matrixmodel} and generate the axial algebras studied in
Section~\ref{sec:axial}.

\begin{definition}\label{def:stardec}
For $\omega\in\Omega$ the \emph{star decoration at $\omega$} is
\begin{equation}\label{stardec}
\delta_\omega(b)=
\begin{cases}
-1, & \omega\in\supp(b),\\
\phantom{-}1, & \omega\notin\supp(b).
\end{cases}
\end{equation}
\end{definition}

\begin{lemma}\label{lem:starkappa}
$\kappa_{\delta_\omega}(P_\rho)=-1$ if and only if $\omega\in\rho$, for
every $\rho\in\binom{\Omega}{4}$.
\end{lemma}

\begin{proof}
By \eqref{d0d1d2}, $\kappa_{\delta_\omega}=d_2\delta_\omega$. Of the four
$3$-subsets of $\rho$, exactly three contain $\omega$ when
$\omega\in\rho$, and none does otherwise; hence the product of the four
values \eqref{stardec} equals $(-1)^3=-1$ in the first case and $+1$ in
the second.
\end{proof}

Thus the star decoration centered at $\omega$ produces the simplest possible Pasch-sign pattern: precisely the Pasch configurations whose support contains $\omega$ are negative.

The parameter $n^{\delta}_{-1}$ admits both a combinatorial and a
spectral interpretation:
$$
n^{\delta}_{-1}
=
|\{r\in R:\kappa_{ijkr}=-1\}|
=
\dim \ker (L(c_b^\delta)+\Id).
$$
Consequently, the local spectral classification of block idempotents
reduces to the distribution of negative Pasch configurations.


We now turn to two natural constructions built on the star decorations: concrete matrix realizations (Section~\ref{sec:matrixmodel}) and axial structure (Section~\ref{sec:axial}).

\section{Matrix models}\label{sec:matrixmodel}

The incidence algebra of $G_2(n)$, and one distinguished decorated
algebra, admit concrete realizations by symmetric matrices. Throughout
this section $\fie$ is a formally real Pythagorean field and $\ualg$ denotes a Euclidean space of dimension $n$ with a
fixed orthonormal basis $v_1,\dots,v_n$, we write
$$
S_{ij}:=v_iv_j^{T}+v_jv_i^{T}\in\Sym(\ualg),
\qquad
i\neq j,
$$
$\Sym_0(\ualg):=\fie\{S_{ij}:i<j\}$ for the symmetric endomorphisms of $\ualg$ whose matrices have vanishing
diagonal in the fixed basis, and $\pi:\Sym(\ualg)\to\Sym_0(\ualg)$ for the projection
along the diagonal (the span of the endomorphisms $v_{i}v_{i}^{T}$).
The Jordan product is written $X\bullet Y=\tfrac12(XY+YX)$, so that the operation
\begin{equation}\label{diamond}
X\odot Y:=\pi(XY+YX)=2\,\pi(X\bullet Y)
\end{equation}
is the projected Jordan product.

\begin{proposition}\label{prop:evenpart}
    The algebra $(\Sym_0(\ualg),\odot)$ is isomorphic to the incidence
    algebra of $G_2(n)$,
\begin{equation}\label{evenpart}
\bigl(\Sym_0(\ualg),\odot\bigr)
\ \cong\
\alg^{\mathbf 1}\bigl(G_2(n)\bigr),
\end{equation}
via the linear map sending $S_{ij}$ to $e_{ij}$.
\end{proposition}

\begin{proof}
Since the stated map is evidently injective, it is a linear isomorphism for dimension reasons.
For distinct $i,j,k$ one has
$S_{ij}S_{ik}=v_jv_k^{T}$ and $S_{ik}S_{ij}=v_kv_i^{T}$, whence
$S_{ij}\odot S_{ik}=\pi(S_{jk})=S_{jk}$. Furthermore
$S_{ij}\odot S_{ij}=2\pi(v_iv_i^{T}+v_jv_j^{T})=0$, and
$S_{ij}\odot S_{kl}=0$ when $\{i,j\}\cap\{k,l\}=\emptyset$. These are
exactly the relations \eqref{multiplication}--\eqref{multiplication2}
for the trivial decoration $\delta\equiv\mathbf 1$.
\end{proof}

    \begin{definition}\label{def:vm}
Denote by $\vm{\ualg}$  the commutative algebra
with underlying vector space
$$\ualg\oplus\Sym_0(\ualg)
$$ and multiplication
\begin{equation}\label{vmproduct}
\begin{aligned}
(x,X)\diamond(y,Y)&=
\bigl(-Xy-Yx,\
\pi\bigl(2X\bullet Y-xy^{T}-yx^{T}\bigr)\bigr)
\end{aligned}
\end{equation}
for $x, y \in \ualg$ and $X, Y \in \Sym_0(\ualg)$.
\end{definition}
Note $\dim\vm{\ualg}=n+\binom n2=\binom{n+1}{2}$.
Restricted to the summands $\ualg$ and $\Sym_0(\ualg)$, \eqref{vmproduct} reads
\begin{align}\label{vmsummands}
\begin{aligned}
x\diamond y&=-\pi\bigl(xy^{T}+yx^{T}\bigr),
\\
x\diamond X&=-Xx,
\\
X\diamond Y&=\pi(XY+YX),
\end{aligned}
\end{align}
so that the map
$X\longmapsto(0,X)$ is an algebra embedding
$$
(\Sym_0(\ualg),\diamond)
\hookrightarrow
(\ualg \oplus \Sym_0(\ualg), \diamond).
$$

\begin{theorem}\label{thm:vm}
Let $m\ge3$, $\omega\in\Omega$, and let $\ualg:=\fie\{\Omega'\}$ for
$\Omega':=\Omega\setminus\{\omega\}$, with orthonormal basis
$(v_i)_{i\in\Omega'}$. Then there is an algebra isomorphism
\begin{equation}\label{vmiso}
(\alg^{\delta_\omega}\bigl(G_2(m)\bigr), \circ_\delta)\ \cong\ (\vm{\ualg}, \diamond),
\end{equation}
where $\delta_\omega$ is a star decoration \eqref{stardec} and an isomorphism being given on the standard basis by
\begin{align}
&e_{\omega i}\longmapsto v_i,
\qquad
e_{ij}\longmapsto S_{ij}
&&
(i,j\in\Omega').
\end{align}
Under \eqref{vmiso}:
\begin{enumerate}
\item[(a)] the block idempotents of the blocks containing $\omega$ are
\begin{equation}\label{vmaxis}
c_{b_{\omega ij}}^{\delta_\omega}
=
-\tfrac12\bigl(u,\ \pi(uu^{T})\bigr),
\qquad
u=v_i+v_j ,
\end{equation}
while the remaining block idempotents lie in $\Sym_0(\ualg)$;
\item[(b)] for $b=b_{\omega ij}$ the involution $\theta_b$ of
Corollary~\ref{cor:gradingauto} is conjugation by the reflection
${\refl}\in O(\ualg)$ fixing $\fie\{v_i,v_j\}$ and acting as $-1$ on its
orthogonal complement:
$\theta_b(x,X)=({\refl}x,\,{\refl}X{\refl})$;
\item[(c)] $\alg^{\delta_\omega}(G_2(m))$ is generated by the block
idempotents $c^{\delta_\omega}_{b}$, $b\ni\omega$; already $m-1$ of them
suffice, namely those indexed by the edges of a Hamiltonian cycle of the
complete graph on $\Omega'$ under $b_{\omega ij}\leftrightarrow\{i,j\}$.
\end{enumerate}
\end{theorem}

\begin{proof}
By \eqref{stardec}, $\delta_\omega(b_{\omega ij})=-1$ and
$\delta_\omega(b_{ijk})=+1$ for $i,j,k\in\Omega'$, so
\eqref{multiplication}--\eqref{multiplication2} give
\begin{align}
\begin{aligned}
e_{\omega i}\circ_{\delta_\omega} e_{\omega j}&=-e_{ij},
\\
e_{\omega i}\circ_{\delta_\omega} e_{ij}&=-e_{\omega j},
\\
e_{\omega i}\circ_{\delta_\omega} e_{jk}&=0&& (i\notin\{j,k\}),
\\
e_{ij}\circ_{\delta_\omega} e_{ik}&=e_{jk},
\end{aligned}
\end{align}
all other products of basis vectors vanishing. On the other side
\eqref{vmsummands} gives
\begin{align}
\begin{aligned}
&v_i\diamond v_j=-S_{ij},&&
v_i\diamond S_{ij}=-v_j,
\\
&v_i\diamond S_{jk}=0&& &&(i\notin\{j,k\}),
\\
&S_{ij}\diamond S_{ik}=S_{jk},
\end{aligned}
\end{align}
together with $v_i\diamond v_i=-2\pi(v_iv_i^{T})=0$ and, by
Proposition~\ref{prop:evenpart}, $S_{ij}\diamond S_{ij}=0$ and
$S_{ij}\diamond S_{kl}=0$ for disjoint index pairs. The two tables agree,
which proves \eqref{vmiso}.

(a) For $u=v_i+v_j$ one has $\pi(uu^{T})=S_{ij}$, so the right-hand side
of \eqref{vmaxis} is the image of
$-\tfrac12(e_{\omega i}+e_{\omega j}+e_{ij})$, which is
$c^{\delta_\omega}_{b_{\omega ij}}$ because
$\delta_\omega(b_{\omega ij})=-1$. A block $b$ with $\omega\notin\supp(b)$
has all its points in $\binom{\Omega'}{2}$, so $c_b^{\delta_\omega}$ maps
into $\Sym_0(\ualg)$.

(b) By Corollary~\ref{cor:gradingauto},
$\theta_b(e_p)=(-1)^{|p\cap\{\omega,i,j\}|}e_p$. Hence
$e_{\omega k}\mapsto\eta_k e_{\omega k}$ and
$e_{kl}\mapsto\eta_k\eta_l e_{kl}$, where $\eta_k=+1$ for $k\in\{i,j\}$
and $\eta_k=-1$ otherwise. Under \eqref{vmiso} this is $x\mapsto {\refl}x$ and
$S_{kl}\mapsto {\refl}S_{kl}{\refl}$ with ${\refl}=\operatorname{diag}(\eta)$.

(c) Let $\malg'\subseteq\vm{\ualg}$ be the subalgebra generated by the
elements \eqref{vmaxis}. For distinct $i,j,k\in\Omega'$,
\eqref{vmproduct} gives
$$
c_{b_{\omega ij}}\diamond c_{b_{\omega ik}}
=
\tfrac12 c_{b_{\omega ij}}+\tfrac12 c_{b_{\omega ik}}+\tfrac12\,v_i ,
$$
so $v_i\in\malg'$ if two of the chosen idempotents share the
index $i$; and then $S_{ij}=-v_i\diamond v_j\in\malg'$ for all
$i,j\in\Omega'$. A family of blocks through $\omega$ produces every
$v_i$ precisely when the corresponding edges of the complete graph on
$\Omega'$ have minimum degree $2$; such a subgraph has at least
$|\Omega'|=m-1$ edges, and a Hamiltonian cycle attains this bound.
\end{proof}

Combining \eqref{vmiso} with Proposition~\ref{prop:evenpart} gives

\begin{corollary}\label{cor:selfsim2}
For every $m\ge3$ and $\omega\in\Omega$, there is an algebra isomorphism
\begin{equation}\label{selfsim2}
\alg^{\delta_\omega}\bigl(G_2(m)\bigr)
\ \cong\
\ualg\ \oplus\ \alg^{\mathbf 1}\bigl(G_2(m-1)\bigr),
\qquad
\ualg=\fie^{\,\Omega\setminus\{\omega\}},
\end{equation}
the two summands being the odd and the even part of the grading
\eqref{taugrading} for $\tau=\{\omega\}$, with $\ualg$ acting on the even
part as its natural module.
\end{corollary}

Thus the star decoration at $\omega$ adjoins the natural module to the
\emph{undecorated} Grassmannian algebra one step down. This is transverse
to the self-similarity of Section~\ref{sec:proof}: there the passage is
$m\mapsto m-3$ and the residual copy of $G_2(m-3)$ is the zero eigenspace
of a block idempotent, whereas here it is $m\mapsto m-1$ and the residual
copy of $G_2(m-1)$ is the even part of a $\mathbb Z_2$-grading.

Assume that $\fie=\R{}$. Then the two smallest instances of \eqref{selfsim2}, at $m=3$ and $m=4$ (see Examples~\ref{ex:polar3} and \ref{ex:mutant} below),
deserve separate comment. Apart from their place in the
$G_2(m)$-hierarchy, $\vm{\mathbb R^2}$ and $\vm{\mathbb R^3}$ are
themselves exact radial Hsiang algebras in the sense of
\cite{Fox-Tkachev,Nadirashvili-Tkachev-Vladuts}, and together the pair
illustrates the concept of \emph{polar}
Hsiang algebras, built from a classical symmetric Clifford system. On the other hand, these are the only Hsiang algebras in the family $G_2(m)$. Indeed, for $m\ge 5$ the eigenvalue $0$ occurs in the Peirce spectrum with positive multiplicity, which is incompatible with the Hsiang property.

Recall that a metrized commutative algebra $(\alg,\scal{\dum}{\dum})$ is called
\emph{polar}, or of \emph{Clifford type},  if it admits an orthogonal $\mathbb Z_2$-grading
$
\alg=\alg_0\oplus\alg_1,
$
where $\alg_0$ is a trivial algebra and $x_0(x_0x_1)=\scal{x_1}{x_1}\,x_0$
for all $x_i\in\alg_i$; see \cite{Fox-Tkachev}. Polar algebras are the best understood subclass of Hsiang algebras. They arise naturally in the construction of algebraic minimal cones via symmetric Clifford systems \cite{Tkachev-cliff}. Symmetric Clifford systems play a central role in the celebrated work of Ferus, Karcher, and M\"unzner \cite{Ferus-Karcher-Munzner} on isoparametric hypersurfaces in spheres. In fact, there is a natural correspondence between symmetric Clifford systems and polar algebras; see \cite[Section~6.5]{Nadirashvili-Tkachev-Vladuts}.

In \cite{Fox-Tkachev}, a distinguished subclass of Hsiang algebras, called \emph{mutants}, was introduced as the triples of classical Hurwitz algebras. The algebras considered in Examples~\ref{ex:polar3} and \ref{ex:mutant} are isomorphic with the triples of the real field $\R{}$ and the complex field $\mathbb C$ viewed as a real algebra, respectively, and therefore carry canonical symmetric Clifford system structures in the sense of \cite{Fox-Tkachev}; see also \cite[Section~6.5]{Nadirashvili-Tkachev-Vladuts}. Moreover, after a suitable normalization, the $6$-dimensional case algebra is isomorphic to a Cayley-Dickson style doubling of the $3$-dimensional cross-product algebra in the sense of \cite{FoxTka26}.

\begin{example}\label{ex:polar3}
Let $\ualg=\mathbb R^2$. Write
$$
x=x_1v_1+x_2v_2,
\qquad
X=
\begin{pmatrix}
0&y_{12}\\
y_{12}&0
\end{pmatrix}.
$$
With respect to the inner product
$$
\langle(x,X),(y,Y)\rangle
=
x^Ty+\tfrac12\tr(XY),
$$
the cubic form of $(\vm{\ualg},\diamond)$ is
$$
u(x,X)
=
\langle(x,X),(x,X)\diamond(x,X)\rangle
=
\tr(X^3)-3x^TXx
=
-6\,x_1x_2y_{12},
$$
since $X^2=y_{12}^2\Id$ is scalar, whence $X^3=y_{12}^2X$ is traceless.
Thus $\vm{\mathbb R^2}$ is the $3$-dimensional metrized algebra whose
cubic form is, up to normalization, the split product $x_1x_2y_{12}$ of
three independent linear forms. Here $\Sym_0(\mathbb R^2)=\fie\{y_{12}\}$
is one-dimensional, and $S_{12}\odot S_{12}=2\pi(v_1v_1^T+v_2v_2^T)=0$ by
\eqref{diamond} (equivalently, $G_2(2)$ has no blocks, so
Proposition~\ref{prop:evenpart} gives the zero algebra); thus
$\Sym_0(\mathbb R^2)$ is a null algebra. Consequently $\vm{\mathbb R^2}$
satisfies the defining axioms of a polar algebra: it is a genuine
\emph{polar Hsiang algebra}, built from the rank-$1$ symmetric Clifford
system on $\mathbb R^2$, in the sense of \cite{Ferus-Karcher-Munzner}. By
Theorem~\ref{thm:vm} (with $m=3$), this is precisely the star-decorated
incidence algebra of $G_2(3)$, whose underlying PSTS is the single block
on three points.
\end{example}

\begin{example}\label{ex:mutant}
Let $\ualg=\mathbb R^3$. Write
$$
x=x_1v_1+x_2v_2+x_3v_3,
\qquad
X=
\begin{pmatrix}
0&y_{12}&y_{13}\\
y_{12}&0&y_{23}\\
y_{13}&y_{23}&0
\end{pmatrix}.
$$
With respect to the inner product
$$
\langle(x,X),(y,Y)\rangle
=
x^Ty+\tfrac12\tr(XY),
$$
the cubic form of $(\vm{\ualg},\diamond)$ is
$$
\begin{aligned}
u(x,X)
&=
\langle(x,X),(x,X)\diamond(x,X)\rangle=
\tr(X^3)-3x^TXx\\
&=
6\bigl(
y_{12}y_{13}y_{23}
-y_{12}x_1x_2
-y_{13}x_1x_3
-y_{23}x_2x_3
\bigr).
\end{aligned}
$$
Introducing complex coordinates
$z_1=y_{12}+\mathrm{i}x_3$, $z_2=y_{13}+\mathrm{i}x_2$, and $z_3=y_{23}+\mathrm{i}x_1$
one obtains
\begin{align}\label{6dmutant}
u(x,X)=6\operatorname{Re}(z_1z_2z_3).
\end{align}
This shows that $\vm{\mathbb R^3}$ is isomorphic to the metrized commutative
algebra associated, up to normalization, with the cubic form \eqref{6dmutant} known in the theory of Hsiang algebras as a \emph{mutant Hsiang
algebra}; see \cite{Nadirashvili-Tkachev-Vladuts,Fox-Tkachev}. By Theorem~\ref{thm:vm}, this is precisely the star-decorated incidence
algebra of $G_2(4)$, for which the underlying PSTS is the Pasch configuration.
\end{example}

\section{Axial structure of the star-decorated algebras}\label{sec:axial}

Recall that an \emph{axial algebra} is a commutative nonassociative
algebra generated by a set of idempotents, its \emph{axes}, each of which
is semisimple with $\Spec L(a)$ contained in a fixed finite set
$\Lambda$, and whose Peirce decomposition obeys a \emph{fusion law}: a
prescribed rule
$
A_\lambda\circ A_\mu\subseteq\bigoplus_{\nu\in\fusion(\lambda,\mu)}A_\nu
$
valid at every axis. The law is $\mathbb Z_2$-graded if
$\Lambda=\Lambda_+\sqcup\Lambda_-$ compatibly with the parity, in which
case the sign map of the grading is an involutive automorphism, the
Miyamoto involution of the axis; axes of Jordan type, with
$\Lambda=\{1,0,\eta\}$, are the classical case
\cite{Hall-Rehren-Shpectorov,Matsuo-3transposition}. By
\eqref{localspectrumformula} the Peirce spectrum of a block idempotent
lies in
$$
\Lambda=\{1,-1,\tfrac12,-\tfrac12,0\},
$$
the sign-doubled analogue of $\{1,0,\tfrac12\}$, which makes the question
meaningful. We show that block idempotents are axes precisely when all
Pasch configurations through the block are negative, and that the star
decorations of Definition~\ref{def:stardec} realize this on a maximal
set of blocks.

\begin{definition}\label{def:starax}
For $\omega\in\Omega$ and a decoration $\delta$ put
$$
\St(\omega):=\{b\in\mathcal B:\ \omega\in\supp(b)\},
\qquad
\Ax(\delta):=\{b\in\mathcal B:\ n^{\delta}_{-1}(b)=m-3\},
$$
so that $|\St(\omega)|=\binom{m-1}{2}$.
\end{definition}

Throughout this section we abbreviate the Peirce eigenspaces of the fixed idempotent
 $c$ by
$$A_\lambda:=\ker\bigl(L(c)-\lambda\,\Id\bigr).
$$

\begin{theorem}\label{thm:grading}
Let $b\in\mathcal B$. Every Peirce eigenspace of $c_b^\delta$ is
homogeneous for the grading \eqref{apriorigrading} if and only if
$b\in\Ax(\delta)$. In that case
\begin{equation}\label{fusiongrading}
\alg_{+}(b)=A_{1}\oplus A_{-1/2}\oplus A_{0},
\qquad
\alg_{-}(b)=A_{-1}\oplus A_{1/2},
\end{equation}
and $\theta_b$ is the sign map of this decomposition.
\end{theorem}

\begin{proof}
By Proposition~\ref{prop:blockdecomp}, $L_\delta(c_b^\delta)$ preserves
\eqref{Vdecomp} and vanishes on $\valg_{\mathcal P_2}$, and by
\eqref{specL0}, \eqref{specplus} its spectra on the remaining summands
are $\{1,(-\tfrac12)^2\}$ on $\valg_{\mathcal P_0}$ and, on
$\valg_{W_r}$, either $\{1,(-\tfrac12)^2\}$ or $\{-1,(\tfrac12)^2\}$
according as $\kappa_{ijkr}=1$ or $-1$. By \eqref{apriorigrading},
$\valg_{\mathcal P_0}\subseteq\alg_+(b)$ while
$\valg_{W_r}\subseteq\alg_-(b)$ for all $r\in R$. If $\kappa_{ijkr}=1$
for some $r$, the eigenvalue $1$ occurs both in $\alg_+(b)$ and in
$\alg_-(b)$, so $A_1$ is not homogeneous; the same applies to
$A_{-1/2}$. If $\kappa_{ijkr}=-1$ for all $r\in R$, then $A_1$ and
$A_{-1/2}$ lie in $\valg_{\mathcal P_0}$, $A_0=\valg_{\mathcal P_2}$ and
$A_{-1},A_{1/2}$ lie in $\bigoplus_r\valg_{W_r}$, which is
\eqref{fusiongrading}. By \eqref{nminusone1} the latter condition is
$n^\delta_{-1}(b)=|R|=m-3$.
\end{proof}

\begin{proposition}\label{prop:star}
Let $m\ge5$ and $\omega\in\Omega$. Then $\Ax(\delta_\omega)=\St(\omega)$.
\end{proposition}

\begin{proof}
The Pasch configurations containing $b$ are indexed by the sets
$\supp(b)\cup\{x\}$, $x\in\Omega\setminus\supp(b)$. If
$\omega\in\supp(b)$, all of them contain $\omega$, hence are negative by
Lemma~\ref{lem:starkappa}, so $n^{\delta_\omega}_{-1}(b)=m-3$. If
$\omega\notin\supp(b)$, choose
$x\in\Omega\setminus(\supp(b)\cup\{\omega\})$, possible because
$m\ge5$; the corresponding configuration is positive, so
$n^{\delta_\omega}_{-1}(b)<m-3$.
\end{proof}

\begin{theorem}\label{thm:vmaxial}
Let $m\ge5$ and $\omega\in\Omega$. Then $\alg^{\delta_\omega}(G_2(m))$ is
an axial algebra with axis set $\{c_b^{\delta_\omega}:b\in\St(\omega)\}$,
the Peirce spectrum of each axis being
$$
\Spec L(c_b)=\bigl\{1,\ (-1)^{m-3},\ (\tfrac12)^{2(m-3)},\
(-\tfrac12)^{2},\ 0^{\binom{m-3}{2}}\bigr\},
$$
and the fusion law being Table~\ref{tab:fusion}, the same at every axis
and independent of $m$.
\end{theorem}

\begin{table}[h]
\centering
\renewcommand{\arraystretch}{1.2}
\begin{tabular}{c|cc|c|cc}
$\diamond$ & $1$ & $-\tfrac12$ & $0$ & $-1$ & $\tfrac12$\\
\midrule
$1$          & $1$ & $-\tfrac12$ & $\varnothing$ & $-1$ & $\tfrac12$\\
$-\tfrac12$  & $-\tfrac12$ & $1,-\tfrac12$ & $\varnothing$ & $\tfrac12$ & $-1,\tfrac12$\\
\midrule
$0$          & $\varnothing$ & $\varnothing$ & $0$ & $-1,\tfrac12$ & $-1,\tfrac12$\\
\midrule
$-1$         & $-1$ & $\tfrac12$ & $-1,\tfrac12$ & $1,0$ & $-\tfrac12,0$\\
$\tfrac12$   & $\tfrac12$ & $-1,\tfrac12$ & $-1,\tfrac12$ & $-\tfrac12,0$ & $1,-\tfrac12,0$\\
\end{tabular}
\smallskip
\caption{The fusion law of a block axis. The eigenvalues are ordered so that the even part of
\eqref{fusiongrading} occupies the diagonal blocks and the odd part the
off-diagonal ones. For $m=5$ the entry $A_0\diamond A_0$ degenerates
to $0$.}
\label{tab:fusion}
\end{table}

\begin{proof}

The spectrum is \eqref{localspectrumformula} with
$n^{\delta_\omega}_{-1}(b)=m-3$, by Proposition~\ref{prop:star}, and
generation is Theorem~\ref{thm:vm}(c). It remains to establish
Table~\ref{tab:fusion}, which we do in the model $\vm{\ualg}$ of
Theorem~\ref{thm:vm}. Fix $b=b_{\omega ij}$ and put
\begin{equation}\label{gdef}
W:=\Omega'\setminus\{i,j\},
\qquad
g:=v_i+v_j+S_{ij},
\qquad
c=c^{\delta_\omega}_b=-\tfrac12 g ,
\end{equation}
and $A_\lambda:=\ker\bigl(L(c)-\lambda\,\Id\bigr)$.

\emph{Step 1: the Peirce decomposition.} Applying \eqref{vmproduct} to
the basis vectors gives
$$
c\diamond v_i=\tfrac12 v_j+\tfrac12 S_{ij},
\qquad
c\diamond v_j=\tfrac12 v_i+\tfrac12 S_{ij},
\qquad
c\diamond S_{ij}=\tfrac12 v_i+\tfrac12 v_j,
$$
and, for $r,s\in W$,
$$
c\diamond v_r=\tfrac12(S_{ir}+S_{jr}),
\quad
c\diamond S_{ir}=\tfrac12 v_r-\tfrac12 S_{jr},
\quad
c\diamond S_{jr}=\tfrac12 v_r-\tfrac12 S_{ir},
\quad
c\diamond S_{rs}=0 .
$$
Hence $L(c)$ preserves
\begin{equation}\label{Ldecomp}
\vm{\ualg}
=
\fie\{v_i,v_j,S_{ij}\}
\oplus
\bigoplus_{r\in W}\fie\{v_r,S_{ir},S_{jr}\}
\oplus
\fie\{S_{rs}:r,s\in W\},
\end{equation}
acting on the first summand by the matrix \eqref{L0matrix}, on the
$r$-th middle summand by
$\tfrac12\left(\begin{smallmatrix}0&1&1\\1&0&-1\\1&-1&0\end{smallmatrix}\right)$,
of characteristic polynomial $-(\lambda-\tfrac12)^2(\lambda+1)$, and by
zero on the last. Setting
\begin{equation}\label{apq}
\begin{aligned}
a_1&:=v_i-v_j,\\
a_2&:=v_i-S_{ij},\\
w_r&:=v_r-S_{ir}-S_{jr},\\
p_r&:=v_r+S_{ir},\\
q_r&:=v_r+S_{jr},
\end{aligned}
\end{equation}
we obtain
\begin{equation}\label{peircebasis}
\begin{aligned}
&A_1=\fie\{g\},
&&
A_{-1/2}=\fie\{a_1,a_2\},
&&
\\
&A_{-1}=\fie\{w_r:r\in W\},
&&
A_{1/2}=\fie\{p_r,q_r:r\in W\},\\
&A_0=\fie\{S_{rs}:r,s\in W\}.
\end{aligned}
\end{equation}
The dimensions $1,2,m-3,2(m-3),\binom{m-3}{2}$ agree with the spectrum,
and \eqref{peircebasis} refines \eqref{fusiongrading}. We shall use the
inverse relations
\begin{equation}\label{inverserel}
S_{ij}=\tfrac13\bigl(g+a_1-2a_2\bigr),
\qquad
v_j=\tfrac13\bigl(g-2a_1+a_2\bigr).
\end{equation}

\emph{Step 2: entries involving $A_1$.} Since $A_1=\fie\{c\}$ and
$c\diamond z=\lambda z$ for $z\in A_\lambda$, one has
$A_1\diamond A_\lambda\subseteq A_\lambda$ for $\lambda\neq0$ and
$A_1\diamond A_0=0$, which is the first row and column of
Table~\ref{tab:fusion}.

\emph{Step 3: $A_0\diamond A_0$.} By \eqref{peircebasis},
$A_0=\Sym_0\bigl(\fie\{v_r:r\in W\}\bigr)$, a subalgebra of
$(\Sym_0(\ualg),\diamond)$ by Proposition~\ref{prop:evenpart}. Hence
$A_0\diamond A_0\subseteq A_0$.

\emph{Step 4: the entries in the $A_{-1/2}$ row.} A direct computation
from \eqref{vmproduct} gives
$$
a_1\diamond a_1=2S_{ij},
\qquad
a_1\diamond a_2=-a_1+S_{ij},
\qquad
a_2\diamond a_2=2v_j ,
$$
so by \eqref{inverserel} all three lie in $A_1\oplus A_{-1/2}$, and
$A_{-1/2}\diamond A_{-1/2}\subseteq A_1\oplus A_{-1/2}$. Next, for
$r,s\in W$,
$$
a_1\diamond S_{rs}=a_2\diamond S_{rs}=0,
$$
whence $A_{-1/2}\diamond A_0=0$. Finally
$$
a_1\diamond w_r=q_r-p_r,
\qquad
a_2\diamond w_r=q_r ,
$$
so $A_{-1/2}\diamond A_{-1}\subseteq A_{1/2}$.

\emph{Step 5: the entries in the $A_{-1}$ row.} For $r\in W$ write
$w_r=(x,X)$ with $x=v_r$ and $X=-(S_{ir}+S_{jr})$; then $Xx=-(v_i+v_j)$,
$\pi(X^{2})=S_{ij}$ and $\pi(xx^{T})=0$, so that \eqref{vmproduct} gives
$$
w_r\diamond w_r=\bigl(2(v_i+v_j),\,2S_{ij}\bigr)=2g\in A_1 .
$$
For $r\neq s$, writing $w_s=(y,Y)$ with $y=v_s$,
$Y=-(S_{is}+S_{js})$, one has $Xy=Yx=0$, $XY+YX=2S_{rs}$ and
$xy^{T}+yx^{T}=S_{rs}$, whence
$$
w_r\diamond w_s=\bigl(0,\pi(2S_{rs}-S_{rs})\bigr)=S_{rs}\in A_0 .
$$
Therefore $A_{-1}\diamond A_{-1}\subseteq A_1\oplus A_0$. Moreover
$$
w_r\diamond p_r=a_2-a_1,
\qquad
w_r\diamond q_r=a_2,
\qquad
w_r\diamond p_s=w_r\diamond q_s=-2S_{rs}\ \ (r\neq s),
$$
so that $A_{-1}\diamond A_{1/2}\subseteq A_{-1/2}\oplus A_0$.

\emph{Step 6: the remaining entries.} By Theorem~\ref{thm:grading} the
decomposition \eqref{fusiongrading} is a $\mathbb Z_2$-grading, so
\begin{align}
\begin{aligned}
A_{-1/2}\diamond A_{1/2},
\
A_{0}\diamond A_{-1},
\
A_{0}\diamond A_{1/2}
\ &\subseteq\
A_{-1}\oplus A_{1/2},\\
A_{1/2}\diamond A_{1/2}
\ &\subseteq\
A_1\oplus A_{-1/2}\oplus A_0 ,
\end{aligned}
\end{align}
which are exactly the corresponding entries of Table~\ref{tab:fusion}.
Together with Steps 2--5 this establishes every entry.

None of the computations depends on the choice of $b\in\St(\omega)$
beyond relabelling of $i,j$ and $W$, and each product involves only the
indices $i,j$ and at most two elements of $W$; hence the law is the same
at every axis and independent of $m$. For $m=5$ one has $|W|=2$, so
$A_0$ is one-dimensional and $A_0\diamond A_0=0$, consistent with Step 3.
\end{proof}

\begin{corollary}\label{cor:evenpartclifford}
For $b=b_{\omega ij}\in\St(\omega)$, the even part
$\alg_+(b)=A_1\oplus A_{-1/2}\oplus A_0$ of \eqref{fusiongrading} splits
as a direct product of algebras
$$
\alg_+(b)\;\cong\;\vm{\mathbb R^2}\ \times\ \alg^1(G_2(m-3)),
$$
where the first factor, spanned by $g,a_1,a_2$ (see \eqref{gdef}, \eqref{apq}), is the  polar
algebra of Proposition~\ref{prop:evenpart} built from the rank-$1$
symmetric Clifford system on $\mathbb R^2$, and the second is the
even-part incidence algebra of Corollary~\ref{cor:selfsim2}, with $m-3$ in place of $m$. Thus for $m\ge5$, $\alg_+(b)$ always has a direct summand which is a polar algebra.
This second factor is nontrivial for $m\ge6$; at the boundary
case $m=5$ it degenerates to a $1$-dimensional null algebra.
\end{corollary}

\begin{proof}
By Step~2 of the proof of Theorem~\ref{thm:vmaxial}, $A_1\diamond A_0=0$,
and by Step~4, $A_{-1/2}\diamond A_0=0$; together with
$A_0\diamond A_0\subseteq A_0$ (Step~3) this shows that $A_0$ is an ideal
of $\alg_+(b)$ on which $A_1\oplus A_{-1/2}$ acts trivially, so
$\alg_+(b)=(A_1\oplus A_{-1/2})\times A_0$ as algebras. The factor
$A_1\oplus A_{-1/2}$ equals $\fie\{g,a_1,a_2\}=\fie\{v_i,v_j,S_{ij}\}$, with the
product inherited from $\vm{\ualg}$. It is exactly the subalgebra
$\vm{\mathbb R^2}\subset\vm{\ualg}$ associated with
$\fie\{v_i,v_j\}\subset\ualg$, hence a  polar algebra by
Proposition~\ref{prop:evenpart}; see also Example~\ref{ex:polar3}. By the same proposition,
$A_0=\Sym_0(\fie\{v_r:r\in W\})\cong\alg^1(G_2(m-3))$. If $m=5$, then
$|W|=2$, so $A_0=\fie\{S_{rs}\}$ for the single pair $r,s\in W$, and
$S_{rs}\diamond S_{rs}=\pi\bigl(2S_{rs}^2\bigr)=\pi\bigl(2(e_re_r^\top+e_se_s^\top)\bigr)=0$
since $\pi$ annihilates diagonal matrices; thus $A_0$ is $1$-dimensional
and null. If $m\ge6$ then $|W|\ge3$, so $G_2(|W|)$ has blocks and
$S_{rs}\diamond S_{rt}=S_{st}\neq0$ for distinct $r,s,t\in W$, so $A_0$
is nontrivial.
\end{proof}

\section{Distribution of spectral types}\label{sec:dist}

Theorem~\ref{thm:localpaschspectrum} reduces the spectral analysis of an individual block idempotent to the single local parameter $n_{-1}^{\delta}(b)$. This naturally leads to a global question: how are these local spectral types distributed throughout the geometry? The purpose of this section is to study the resulting distribution of spectral types and the associated spectral profile.

The next two definitions make sense for an arbitrary decorated
partial Steiner triple system.

\begin{definition}
For a decorated PSTS $(\mathcal P,\mathcal B,\delta)$, let
\begin{equation}\label{Ndelta}
\mathcal N_\delta
=
\{
P:\kappa_\delta(P)=-1
\}
\end{equation}
be the collection of \emph{negative Pasch configurations} of the
decorated system, and, for a block $b\in\mathcal B$, let
\begin{equation}\label{degreegeneral}
d_\delta(b)
=
|\{P\in\mathcal N_\delta:\ b\in P\}|
\end{equation}
denote the number of negative Pasch configurations containing $b$.
\end{definition}

\begin{definition}
For each integer \(k\ge 0\), let
\begin{equation}\label{Tdelta}
T_k^\delta
=
\{\,b\in\mathcal B:d_\delta(b)=k\,\}.
\end{equation}
The vector
\begin{equation}\label{Tdelta1}
\mathbf T_\delta
=
\bigl(|T_0^\delta|,|T_1^\delta|,\dots\bigr)
\end{equation}
is called the \emph{spectral profile} of the decorated partial Steiner
triple system \((\mathcal P,\mathcal B,\delta)\).
\end{definition}

We now specialize to the Grassmannian geometry $G_2(m)$.
By
Theorem~\ref{thm:localpaschspectrum}, the multiplicity of the eigenvalue
$-1$ of the block idempotent $c_b^\delta$ is precisely the number of
negative local Pasch configurations through the chosen block;
equivalently,
\begin{equation}\label{neqd}
d_\delta(b)
=
n_{-1}^{\delta}(b)
=
\dim A_{-1}(c_b^\delta),
\end{equation}
where $A_{-1}(c_b^\delta)$ denotes the $(-1)$-eigenspace of
$L_{\delta}(c_b^\delta)$. Consequently, in the notation \eqref{sigmanot},
$$
T_k^\delta
=
\{
b\in\mathcal B:
\Spec L_\delta(c_b^\delta)=\sigma_k
\},
$$
so that $T_k^\delta$ is the class of all blocks whose associated block
idempotents have spectral type $\sigma_k$, and the spectral profile
records the distribution of these types throughout the geometry.

Since every Pasch configuration of $G_2(m)$ is uniquely determined by a
$4$-subset of $\Omega$, $\mathcal N_\delta$ is identified with
\begin{equation}\label{Ndeltagrass}
\mathcal N_\delta
=
\{\rho\in\binom{\Omega}{4}:
\kappa_\delta(P_\rho)=-1\}.
\end{equation}
The elements of $\mathcal N_\delta$ are called \emph{negative Pasch configurations}. For a block
$b\in\mathcal B$ with support $\supp(b) = \{i,j,k\}$ and for
$r\in\Omega\setminus\{i,j,k\}$, write
$$
\kappa_{b,r}
:=
\kappa_{ijkr}.
$$
Rewriting the definition \eqref{neqd} in these notations yields
$d_\delta(b) = |\{r\in\Omega\setminus\{i,j,k\}:\kappa_{b,r}=-1\}|$, from which it is apparent that $\kappa_{b,r}=-1$ precisely when $\rho=\supp(b)\cup\{r\}$ belongs to $\mathcal N_\delta$. It follows that
\begin{equation}\label{degreeformula}
d_\delta(b)
=
|\{\rho\in\mathcal N_\delta:\supp(b)\subset\rho\}|.
\end{equation}
Equation~\eqref{degreeformula} suggests viewing negative Pasch configurations as
hyperedges and blocks as vertices.

\begin{definition}
The \emph{negative Pasch hypergraph} of the decorated Grassmannian
geometry $(G_2(m),\delta)$ is the $4$-uniform hypergraph
\begin{equation}
\mathscr H_\delta
=
\bigl(
\mathcal B,
\mathcal N_\delta
\bigr),
\end{equation}
whose vertex set $\mathcal B$ is the set of blocks of $G_2(m)$ and whose hyperedges
are the negative Pasch configurations \eqref{Ndeltagrass}, a hyperedge $\rho$ being
incident with a block $b$ whenever $\supp(b)\subset\rho$.
\end{definition}

Each hyperedge $\rho=\{i,j,k,r\}\in\binom{\Omega}{4}$ is incident with
precisely the four blocks whose support is a $3$-subset of $\rho$,
namely those with support
\begin{equation}
\{i,j,k\},\qquad \{i,j,r\},\qquad \{i,k,r\},\qquad \{j,k,r\},
\end{equation}
so that \eqref{degreeformula} says that $d_\delta(b)$ is exactly
the degree of the vertex $b$ in $\mathscr H_\delta$. Thus
$d_\delta(b)$ carries three equivalent interpretations: a spectral
one, by \eqref{neqd}; a geometric one, as the number of negative Pasch
configurations through $b$; and a hypergraph-theoretic one, as a
vertex degree. In particular, since $|T_k|$ is the number of vertices of
degree $k$ in $\mathscr H_\delta$, the spectral profile
\begin{equation}
\mathbf T_\delta
=
\bigl(|T_0|,|T_1|,\dots,|T_{m-3}|\bigr)
\end{equation}
coincides with the degree distribution of the hypergraph
$\mathscr H_\delta$.
A restriction on the possible spectral profiles is given by the
following binomial moment identities.

In the Grassmannian geometry $G_2(m)$, two Pasch configurations
$P_\rho$ and $P_{\rho'}$ contain a common block if and only if $|\rho\cap\rho'|=3$, in which case $\rho\cap\rho'$ is the support of the common block.
More generally, collections of negative Pasch configurations sharing a
common block lead to the following family of overlap invariants.

\begin{definition}\label{deltageneral}
For $p\ge 0$, define
\begin{align}\label{deltageneraldefined}
q_\delta^{(p)}
=
\Bigl|
\Bigl\{
S\subseteq\mathcal N_\delta:
|S|=p,
\quad
\Bigl|\bigcap_{\rho\in S}\rho\Bigr|\geq 3
\Bigr\}
\Bigr|.
\end{align}
When $p = 0$ it is understood that the intersection condition in \eqref{deltageneraldefined} is satisfied vacuously, so $q_\delta^{(0)}
=
\binom{m}{3}$, while when $p = 1$, $q_\delta^{(1)}=4|\mathcal N_\delta|$. When $p \geq 2$, then the condition $\geq 3$ in \eqref{deltageneraldefined} can be replaced by equality, as the $\rho$ corresponding to two distinct Pasch configurations cannot coincide.
\end{definition}

Thus $q_\delta^{(p)}$ counts unordered collections of $p$ negative
Pasch configurations sharing a common block. The cases $p= 0$, $p= 1$, and $p =2$ correspond respectively to the total number of blocks, the handshaking
identity, and pairs of negative Pasch configurations sharing a common
block.

\begin{theorem}\label{thm:highermoments}
For every decoration $\delta$ of $G_2(m)$ and every integer $0\le p\le m-3$, there holds
\begin{equation}\label{momentgeneral}
\sum_{i=0}^{m-3}\binom{i}{p}|T_i|
=
q_\delta^{(p)}.
\end{equation}
In particular,
\begin{align}
&\sum_{k=0}^{m-3}|T_k|=
\binom{m}{3},\label{Ncount}\\
&\sum_{k=1}^{m-3}k\,|T_k|=
4|\mathcal N_\delta|,
\label{Handshake}\\
&\sum_{k=2}^{m-3}k(k-1)|T_k|=
2q_\delta^{(2)}.
\end{align}
\end{theorem}

\begin{proof}
For a block $b\in\mathcal B$, the degree $d_\delta(b)$ equals the number of negative Pasch configurations containing $b$. Hence $$ \binom{d_\delta(b)}{p} $$ counts unordered collections of $p$ distinct negative Pasch configurations containing $b$ (and vanishes if $d_\delta(b)<p$). Summing over all blocks yields $$ \sum_{b\in\mathcal B} \binom{d_\delta(b)}{p} = \sum_{i=0}^{m-3} \binom{i}{p}|T_i|, $$ since each block in $T_i$ contributes exactly $\binom{i}{p}$ such collections.
For $p=0$, this is simply the number of blocks,
$\binom{m}{3}.$
For $p=1$, each negative Pasch contains exactly four blocks, hence
$$
\sum_{b\in\mathcal B}d_\delta(b)
=
4|\mathcal N_\delta|.
$$
Assume now that $p\ge2$. Since each block of $G_2(m)$ is a $3$-subset
of $\Omega$, any collection of $p$ negative Pasch configurations
containing a common block determines a unique element of
$$
\Bigl\{
S\subseteq\mathcal N_\delta:
|S|=p,
\quad
\Bigl|\bigcap_{\rho\in S}\rho\Bigr|=3
\Bigr\}.
$$
Conversely, every such collection contributes exactly once. Therefore
$$
\sum_{b\in\mathcal B}
\binom{d_\delta(b)}{p}
=
q_\delta^{(p)},
$$
which proves \eqref{momentgeneral}.
\end{proof}

\begin{corollary}\label{cor:spectraloverlap}
The spectral profile is uniquely determined by the overlap invariants
\begin{align}\label{overlapinvariants}
q_\delta^{(0)},
q_\delta^{(1)},
\dots,
q_\delta^{(m-3)},
\end{align}
via the explicit inversion formulas
\begin{equation}\label{momentgeneralinverse}
|T_i|
=
\sum_{p=i}^{m-3}
(-1)^{p-i}
\binom{p}{i}
q_\delta^{(p)},
\qquad
0\le i\le m-3.
\end{equation}
\end{corollary}

\begin{proof}
By Theorem~\ref{thm:highermoments}, $q_\delta^{(p)}=\sum_j\binom{j}{p}|T_j|$.
Hence, for $0\le i\le m-3$,
\begin{align*}
\sum_{p=i}^{m-3}(-1)^{p-i}\binom{p}{i}q_\delta^{(p)}
&=
\sum_{j=0}^{m-3}|T_j|
\sum_{p=i}^{m-3}(-1)^{p-i}\binom{p}{i}\binom{j}{p}
\\
&=
\sum_{j=i}^{m-3}|T_j|\binom{j}{i}
\sum_{p=i}^{j}(-1)^{p-i}\binom{j-i}{p-i}
\\
&=
\sum_{j=i}^{m-3}|T_j|\binom{j}{i}
\sum_{k=0}^{j-i}(-1)^{k}\binom{j-i}{k}
\\
&=
\sum_{j=i}^{m-3}|T_j|\binom{j}{i}(1-1)^{j-i}
\;=\;
|T_i| .
\end{align*}
In the second equality there is used $\binom{j}{p}=0$ for $p>j$, which both
removes the terms with $j<i$ and truncates the inner sum at $p=j$,
together with the identity
$\binom{j}{p}\binom{p}{i}=\binom{j}{i}\binom{j-i}{p-i}$; in the third equality there is
substituted $k=p-i$. Each summand in the last sum vanishes except when $j=i$, in which case it
equals $1$.
\end{proof}

Corollary \ref{cor:spectraloverlap} shows that the spectral profile and the overlap invariants contain the same information: the
local degree distribution can be reconstructed entirely from the overlap
statistics of negative Pasch configurations. In practice the inversion
proceeds triangularly from the top degree downward, since the highest
overlap moments receive contributions from the highest spectral classes
only. Indeed,
$$
q_\delta^{(m-3)}
=
|T_{m-3}|,
$$
while
$$
q_\delta^{(m-4)}
=
|T_{m-4}|
+
(m-3)|T_{m-3}|,
$$
whence
$$
|T_{m-4}|
=
q_\delta^{(m-4)}
-
(m-3)q_\delta^{(m-3)}.
$$
At the other end of the range,
$$
|T_0|
=
q_\delta^{(0)}
-
q_\delta^{(1)}
+
q_\delta^{(2)}
-
q_\delta^{(3)}
+\cdots,
\qquad
|T_1|
=
q_\delta^{(1)}
-
2q_\delta^{(2)}
+
3q_\delta^{(3)}
-\cdots.
$$
The difficulty therefore lies not in reconstructing the profile from the
moments, but rather in understanding which collections of invariants
$$
q_\delta^{(0)},
q_\delta^{(1)},
\dots,
q_\delta^{(m-3)}
$$
can actually occur for decorations of $G_2(m)$.

\begin{remark}
The inversion formula allows one to reconstruct the spectral profile
from the overlap invariants. In the smallest case $m=4$, one has
$q_\delta^{(0)}=4$ and $q_\delta^{(1)}=4|\mathcal N_\delta|,$
 therefore
$$
|T_1|
=
4|\mathcal N_\delta|,
\qquad
|T_0|
=
4(1-|\mathcal N_\delta|).
$$
Since the profile coordinates must be nonnegative, it follows that
$$
|\mathcal N_\delta|\in\{0,1\},
$$
yielding two spectral profiles,
$(4,0)$ and $(0,4)$ (see Example \ref{ex:m4}).

For $m=5$, the inversion formula together with a similar argument imply $18$ integer solutions $(|\mathcal N_\delta|,q_\delta)$.
However, only three spectral profiles actually occur for decorations of $G_2(5)$ (see Example \ref{ex:m5hypergraph}). Thus the moment identities are very far from characterizing
realizability. Even in the smallest nontrivial examples they provide
necessary conditions for the existence of a spectral profile, but not
sufficient ones.
\end{remark}

\section{Examples}\label{sec:exa}
Theorem~\ref{thm:localpaschspectrum} reduces the spectral classification of block idempotents to the single parameter $n_{-1}^{\delta}(b)$, the corresponding spectra being the spectral types $\sigma_k$ introduced in \eqref{sigmanot}. The following examples illustrate this classification for the smallest Grassmannian geometries and show how these types emerge in concrete cases.

\begin{example}[The case $m = 2$]
The geometry $G_2(2)$ consists of a single point and contains no blocks. The associated algebra is therefore trivial.
\end{example}

\begin{example}[The case $m= 3$]
The geometry $G_2(3)$ consists of a single block. The two possible decorations of $G_2(3)$ are gauge equivalent so both decorated incidence algebras are isomorphic with the incidence algebra of the trivial decoration. In this case $R=\varnothing$ and
$n^{\delta}_{-1}=0$. Theorem \ref{thm:localpaschspectrum} yields
$$
\Spec(L_\delta(c_b^\delta))=\sigma_0
=
\{1,(-\tfrac12)^2\}.
$$
The $3$-dimensional incidence algebra associated with $G_2(3)$ provides the $3$-dimensional building block basic to the present approach. In particular, this is precisely the spectrum of the distinguished block idempotent in the one-block algebra generated by $b=\{ij,ik,jk\}$,
which occurs repeatedly as the local component $\valg_{\mathcal{P}_0}$ and the triangle modules $\valg_{W_r}$ in the decomposition of larger Grassmannian geometries.
\end{example}

\begin{example}[$m=4$: Pasch configurations]\label{ex:m4}
For $m=4$, there holds $|R|=1$, so $n^{\delta}_{-1}\in\{0,1\}$, and, because $|G_2(4)|=\binom42=6$, the decorated incidence algebra $\alg^\delta(G_2(4))$ is $6$-dimensional.
The two possible spectra are
\begin{align*}
&\sigma_0 =\{1^2,(-\tfrac12)^4\},&&\sigma_1=\{1,(-1),(-\tfrac12)^2,(\tfrac12)^2\}.
\end{align*}
In particular, $m=4$ is the first case where the eigenvalue $-1$ can occur.

The spectrum $\sigma_1$ can be viewed as the Peirce spectrum of the elementary $6$-dimensional Hsiang-type subalgebra generated by a block together with its unique adjacent triangle module \cite{Fox-Tka2026b}.
Because any two blocks are incident, the decomposition
$G_2(4) =
\mathcal{P}_0
\sqcup
W_r$ contains no residual part $\mathcal{P}_2$. Consequently, every spectral contribution comes from the two copies of the local geometry $G_2(3)$, namely the distinguished block $\mathcal{P}_0$ and the unique triangle module $W_r$.
In particular, the two possible spectra $\sigma_{0}$ and $\sigma_{1}$ correspond with the two possible signs
$$
\kappa_{ijkr}=\pm1
$$
of the unique Pasch configuration through the fixed block.
\end{example}

\begin{example}[$m = 5$: the Desargues configuration as a Pasch hypergraph]
\label{ex:m5hypergraph}

The geometry $G_2(5)$ has $|G_2(5)|=\binom52=10$ points and $|\mathcal B|=\binom53=10$ blocks.
Figure~\ref{fig:G25paschhypergraph} illustrates the situation for the
Desargues configuration $G_2(5)$. The ten vertices of the figure correspond to the ten blocks of
$G_2(5)$.
Each colored tetrahedron represents one of the five Pasch
configurations of $G_2(5)$. Precisely, the tetrahedron labelled
$P_{ijkl}$ has as its vertices the four blocks $ijk$, $ijl$, $ikl$, and $jkl$.

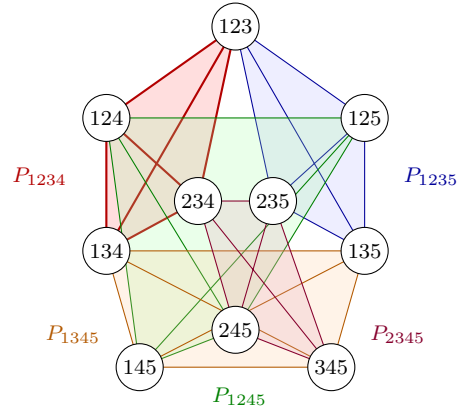
\begin{figure}[h]
\centering
\begin{tikzpicture}[scale=1.1,line join=round,line cap=round,every node/.style={font=\scriptsize}]
\coordinate (b123) at (0,2.35);
\coordinate (b124) at (-1.55,1.25);
\coordinate (b125) at (1.55,1.25);
\coordinate (b134) at (-1.55,-0.35);
\coordinate (b135) at (1.55,-0.35);
\coordinate (b145) at (-1.15,-1.75);
\coordinate (b234) at (-0.45,0.25);
\coordinate (b235) at (0.45,0.25);
\coordinate (b245) at (0,-1.30);
\coordinate (b345) at (1.15,-1.75);
\fill[red!35,opacity=.20] (b123)--(b124)--(b134)--cycle;
\fill[red!35,opacity=.20] (b123)--(b124)--(b234)--cycle;
\fill[red!35,opacity=.20] (b123)--(b134)--(b234)--cycle;
\fill[red!35,opacity=.20] (b124)--(b134)--(b234)--cycle;
\draw[red!70!black,thick] (b123)--(b124)--(b134)--(b123) (b123)--(b234)--(b124) (b134)--(b234);
\fill[blue!30,opacity=.12] (b123)--(b125)--(b135)--cycle;
\fill[blue!30,opacity=.12] (b123)--(b125)--(b235)--cycle;
\fill[blue!30,opacity=.12] (b123)--(b135)--(b235)--cycle;
\fill[blue!30,opacity=.12] (b125)--(b135)--(b235)--cycle;
\draw[blue!60!black] (b123)--(b125)--(b135)--(b123) (b123)--(b235)--(b125) (b135)--(b235);
\fill[green!35,opacity=.12] (b124)--(b125)--(b145)--cycle;
\fill[green!35,opacity=.12] (b124)--(b125)--(b245)--cycle;
\fill[green!35,opacity=.12] (b124)--(b145)--(b245)--cycle;
\fill[green!35,opacity=.12] (b125)--(b145)--(b245)--cycle;
\draw[green!50!black] (b124)--(b125)--(b145)--(b124) (b124)--(b245)--(b125) (b145)--(b245);
\fill[orange!40,opacity=.12] (b134)--(b135)--(b145)--cycle;
\fill[orange!40,opacity=.12] (b134)--(b135)--(b345)--cycle;
\fill[orange!40,opacity=.12] (b134)--(b145)--(b345)--cycle;
\fill[orange!40,opacity=.12] (b135)--(b145)--(b345)--cycle;
\draw[orange!70!black] (b134)--(b135)--(b145)--(b134) (b134)--(b345)--(b135) (b145)--(b345);
\fill[purple!35,opacity=.12] (b234)--(b235)--(b245)--cycle;
\fill[purple!35,opacity=.12] (b234)--(b235)--(b345)--cycle;
\fill[purple!35,opacity=.12] (b234)--(b245)--(b345)--cycle;
\fill[purple!35,opacity=.12] (b235)--(b245)--(b345)--cycle;
\draw[purple!70!black] (b234)--(b235)--(b245)--(b234) (b234)--(b345)--(b235) (b245)--(b345);
\foreach \v/\lab in {
b123/123,
b124/124,
b125/125,
b134/134,
b135/135,
b145/145,
b234/234,
b235/235,
b245/245,
b345/345}
{
\node[circle,draw,fill=white,inner sep=1.5pt] at (\v) {$\lab$};
}
\node[red!70!black] at (-2.35,0.55) {$P_{1234}$};
\node[blue!60!black] at (2.35,0.55) {$P_{1235}$};
\node[green!50!black] at (0.05,-2.10) {$P_{1245}$};
\node[orange!70!black] at (-1.95,-1.35) {$P_{1345}$};
\node[purple!70!black] at (1.95,-1.35) {$P_{2345}$};
\end{tikzpicture}
\caption{The five Pasch tetrahedra of $G_2(5)$. The ten vertices are the blocks of $G_2(5)$, and each colored tetrahedron represents one Pasch configuration.}
\label{fig:G25paschhypergraph}
\end{figure}

A decoration induces a sign on each Pasch configuration. Geometrically,
this amounts to declaring each tetrahedron positive or negative.
Consequently, $n_{-1}^{\delta}(b)$ is simply the number of negative tetrahedra incident with the vertex
$b$. Thus the spectral type of a block is determined by the local
incidence pattern of negative tetrahedra around the corresponding
vertex. For $m=5$, every block belongs to exactly two Pasch configurations, hence
$$
0\le n_{-1}^{\delta}(b)\le 2.
$$
This immediately explains why exactly three spectral types occur in
$G_2(5)$.

Figure~\ref{fig:G25paschhypergraph} also admits a natural interpretation in terms of the regular
$4$-simplex. The five Pasch configurations correspond to its five
tetrahedral facets, while the ten blocks correspond to its ten
triangular $2$-faces. In this model,
$n_{-1}^{\delta}(b)$ counts the number of negative tetrahedral
facets incident with the triangular face corresponding to $b$.

Since
$|R|=2$, Corollary~\ref{cor:realizable} implies that
$n^{\delta}_{-1}\in\{0,1,2\}$ and there occur exactly three spectral types, namely the types listed in
Table~\ref{tab:G25spectra}.

\begin{table}[h]
\centering
\begin{tabular}{c|l}
$n^{\delta}_{-1}$ &
$\Spec(L_\delta(c_b^\delta))$
\\
\midrule
$0$
&
$\sigma_0=\{0,1^3,(-\tfrac12)^6\}$
\\
$1$
&
$\sigma_1=\{0,1^2,(-1),(-\tfrac12)^4,(\tfrac12)^2\}$
\\
$2$
&
$\sigma_2=\{0,1,(-1)^2,(-\tfrac12)^2,(\tfrac12)^4\}$
\\
\end{tabular}
\caption{Possible spectra of block idempotents in $G_2(5)$.}
\label{tab:G25spectra}
\end{table}

The case $m=5$ occupies a distinguished position in the Grassmannian
family. Indeed,
$|\mathcal{P}_2| = \binom{m-3}{2} = 1$, so that $G_2(5)$ is the smallest member of the family in which the
residual Grassmannian appears, and it does so with multiplicity
one. Thus the block spectrum is still almost entirely governed by the
local Pasch data. Starting with $m=6$, the residual contribution
$|\mathcal{P}_2| = \binom{m-3}{2}$ grows quadratically, and the eigenvalue $0$ becomes an essential
component of the spectrum.

This example concludes with a concrete illustration of the first and second moment
identities. Consider the decoration such that $\mathcal N_\delta=\{1234\}$, so that the tetrahedron corresponding to $P_{1234}$ is the unique
negative Pasch in Figure~\ref{fig:G25paschhypergraph}. In the Pasch hypergraph the four vertices $123$, $124$, $134$, and $234$ have degree $1$, while the remaining six vertices have degree $0$.
Hence $|T_0|=6$, $|T_1|=4$, and $|T_2|=0$, so that the spectral profile is
$
\mathbf T_\delta=(6,4,0).
$
Moreover,
$$
0\cdot |T_0|
+
1\cdot |T_1|
+
2\cdot |T_2|
=
4
=
4|\mathcal N_\delta|,
$$
as guaranteed by the handshake identity \eqref{Handshake} of Theorem~\ref{thm:highermoments}. Since there is only one negative Pasch, there are no pairs of distinct
negative Pasch configurations sharing a block. Thus
$q_\delta^{(2)}=0,$ and
$$
0^2\cdot |T_0|
+
1^2\cdot |T_1|
+
2^2\cdot |T_2|
=
4
=
4|\mathcal N_\delta|+2q_\delta^{(2)}.
$$
\end{example}

\section{Concluding remarks and observations}\label{sec:concl}

The results obtained in this paper suggest several directions for
further investigation. Rather than pursuing them here, we conclude by
summarizing a number of observations related
to the realization problem for spectral profiles and which merit more
systematic study elsewhere.

Recall the notion of gauge equivalence of decorations described in
the introduction. A sign function
$\varepsilon:\mathcal P\to\{\pm1\}$ on points transforms a decoration
$\delta$ into
$$
\delta^\varepsilon(\{p,q,r\})
=
\varepsilon(p)\varepsilon(q)\varepsilon(r)\delta(\{p,q,r\}),
$$
and two decorations are gauge equivalent if they are related in this way. A
direct computation shows that the Pasch signs $\kappa_\delta$ defined in \eqref{kappaP} are
invariant under this action. Consequently, gauge equivalent decorations
determine the same collection $\mathcal N_\delta$ of negative Pasch
configurations, hence the same negative-Pasch hypergraph
$\mathscr H_\delta$, and therefore the same spectral profile. Thus the
spectral invariants considered here factor through the
quotient of the decoration space by gauge equivalence, and the
realization problem for spectral profiles is naturally formulated for
gauge-equivalence classes of decorations rather than for individual decorations.

Combining the bijection \eqref{rankd2} between Pasch-sign patterns and
gauge-equivalence classes with the rank computation \eqref{rankd1}, both
established in Section~\ref{sec:d-complex}, the number of
gauge-equivalence classes of decorations of $G_2(m)$ -- called the
\emph{gauge number} of $G_2(m)$ and denoted by $\gamma(G_2(m))$ --
equivalently the number of realizable Pasch-sign patterns -- is
\begin{equation}\label{gaugecount}
\gamma(G_2(m))
=
2^{\binom{m}{3}-\binom{m-1}{2}}
=
2^{\binom{m-1}{3}}.
\end{equation}

The spectral profile introduced in Section~\ref{sec:dist} records the
global distribution of local spectral types or, equivalently, the degree
distribution of $\mathscr H_\delta$. A convenient way to encode this
information is the generating polynomial
\begin{equation}
F_\delta(t)
=
\sum_{k\ge0}|T_k|t^k.
\end{equation}
By the moment identities of Section~\ref{sec:dist}, the spectral profile is
equivalent to the collection of overlap invariants \eqref{overlapinvariants}.
Thus the realization problem may be viewed either as the problem of
describing all realizable spectral profiles or, equivalently, as the
problem of characterizing all realizable overlap invariants modulo gauge
equivalence.

These considerations acquire additional content once the symmetries of
the geometry itself are taken into account. The automorphism group of
$G_2(m)$, that is, the group of permutations of the points preserving
the set of blocks, is isomorphic with the permutation group of the underlying $m$-set:
\begin{equation}\label{autsm}
\Aut(G_2(m))
\cong
S_m,
\qquad
m\ge3.
\end{equation}
A permutation of $\Omega$ acts on points and blocks via its induced actions
on $2$-subsets and $3$-subsets. The induced homomorphism from $S_m$ to $\Aut(G_2(m))$ is injective for $m\ge3$.
For surjectivity, it
suffices to note that the collinearity graph of $G_2(m)$, in which two
points are adjacent whenever they lie in a common block, is the Johnson
graph $J(m,2)$, whose automorphism group is $S_m$ for all $m\ne4$
\cite{Ramras-Donovan,Ganesan,Brouwer-Cohen-Neumaier}; since an
automorphism of the geometry preserves collinearity, it must lie
in $S_m$. In the exceptional case $m=4$ the graph $J(4,2)$ is the
octahedron and admits in addition the complementation
$ij\mapsto kl$, but this map sends the block $\{ij,ik,jk\}$ to
$\{kl,jl,il\}$, which is not a block, so that \eqref{autsm} persists.
For $m=5$ the identity \eqref{autsm} is the classical statement that the
Desargues configuration has automorphism group $S_5$ of order $120$
\cite{Knarr-Stroppel-Stroppel}.

The automorphism and gauge equivalence groups act on the decoration space in complementary ways, and it
is their combined action that governs the classification. The cochain
complex above is $S_m$-equivariant: the group acts on each $C^k$ by
permuting subsets of $\Omega$, and every coboundary operator $d_i$
commutes with this action. In particular $\operatorname{im}d_1$ is
$S_m$-invariant, so that $S_m$ acts on the quotient
$
C^2/\operatorname{im}d_1
\cong
\operatorname{im}d_2,
$
that is, on the set of $\gamma(G_2(m))$ gauge classes counted by
\eqref{gaugecount}. Two decorations lying in the same orbit differ by a
gauge transformation followed by an automorphism of the geometry; they
therefore determine isomorphic decorated incidence algebras and, in
particular, the same spectral profile. Write $N_m$ for the number of
$S_m$-orbits on gauge classes, which by Burnside's lemma equals
\begin{equation}\label{burnside}
N_m
=
\frac1{m!}
\sum_{\sigma\in S_m}
2^{\dim\operatorname{Fix}_\sigma(\operatorname{im}d_2)},
\end{equation}
and $a_m$ for the number of distinct spectral profiles of $G_2(m)$. Passing from decorations to gauge classes, then to isomorphism types, and
finally to spectral profiles yields orbit spaces having successive dimensions:
\footnote{The first reduction in \eqref{chain} is drastic, and the second is the
passage to orbits under a group of order $m!$. Asymptotically
$N_m\sim\gamma(G_2(m))/m!$, the identity term dominating the Burnside sum
\eqref{burnside}, although for the small values collected in
Table~\ref{tab:profilecounts} the action is far from free and $N_m$
exceeds $\gamma(G_2(m))/m!$ by factors of $22.5$, $11.3$ and $2.0$ for
$m=5,6,7$ respectively. The third reduction is of a different nature.}
\begin{equation}\label{chain}
2^{\binom m3}
>
\gamma(G_2(m))
\ge N_m
\ge a_m.
\end{equation}
Computational experiments indicate that the set of realizable spectral
profiles is highly constrained. As explained in Example \ref{ex:m5hypergraph}, for $G_2(5)$ there are only three
profiles, with generating polynomials
$$
F_\delta(t)
=
10,
\qquad
3+6t+t^2,
\qquad
4t+6t^2.
$$
Table~\ref{tab:profilecounts} exhibits the values of
\eqref{chain} for the smallest Grassmannian geometries.

\begin{table}[h]
\centering
\begin{tabular}{c|ccc}
Geometry &
$\gamma(G_2(m))$ &
$N_m$ &
$a_m$
\\
\midrule
$G_2(4)$ & $2^{1}=2$ & $2$ & $2$ \\
$G_2(5)$ & $2^{4}=16$ & $3$ & $3$ \\
$G_2(6)$ & $2^{10}=1024$ & $16$ & $16$ \\
$G_2(7)$ & $2^{20}=1048576$ & $423$ & $185$
\end{tabular}
\smallskip
\caption{The three reductions of \eqref{chain}
}
\label{tab:profilecounts}
\end{table}

The three columns of Table~\ref{tab:profilecounts} correspond to three
successive acts of forgetting. Passing from decorations to gauge classes
forgets a choice of basis, and passing from gauge classes to their
$S_m$-orbits forgets the names of the elements of $\Omega$; neither step
discards anything about the object itself, so $\gamma(G_2(m))$ and $N_m$
are merely increasingly honest counts of the genuinely distinct decorated
algebras. The spectral profile is different: it is the degree
distribution of $\mathscr H_\delta$, and a degree distribution measures a
hypergraph rather than describing it. It is therefore the comparison of
$N_m$ with $a_m$ --- not of $\gamma(G_2(m))$ with $a_m$ --- that carries
the real information.

For $m \leq 6$, $N_m$ and $a_m$ are equal. In this range the degree data of the
negative-Pasch hypergraph happen to determine the hypergraph itself, so
that the spectral profile is a \emph{complete} invariant; it separates
decorated incidence algebras up to gauge equivalence and symmetry of the
underlying geometry.
At $G_2(7)$ this fails, $423$ orbits yields only
$185$ distinct spectral profiles. Thus $m=7$ is the first value
at which the spectral profile ceases to be complete, and it is also the
first value for which the residual Grassmannian $G_2(m-3)$ carries a Pasch
configuration of its own; whether the two phenomena are related is not
clear. For $m=8$ exhaustive enumeration becomes impractical, since
by \eqref{gaugecount} there are already $\gamma(G_2(8))=2^{35}$ Pasch-sign
patterns to inspect; random sampling nevertheless produces thousands of
distinct profiles, showing that $a_m$ continues to grow rapidly.

At present no conceptual explanation of these numbers is available. The
preceding discussion suggests several natural questions. Is there an
intrinsic characterization of the realizable overlap invariants \eqref{overlapinvariants}?
Which additional algebraic or combinatorial constraints distinguish
realizable profiles from formally admissible solutions of the moment
identities?
Finally, what accounts for the
collapse of $S_m$-orbits onto spectral profiles beginning at $m=7$, and
is there an intrinsic invariant refining the spectral profile which
remains complete for all $m$?

A different direction concerns how much of the theory depends on Pasch
configurations specifically. By the block decomposition of
Section~\ref{sec:proof}, the points one Pasch configuration away from a
block $b$ --- namely $b$ itself together with the petals $W_r$, $r\in
R$ --- already account for the entire nonzero Peirce spectrum of the
block idempotent, the residual copy of $G_2(m-3)$ contributing only the zero eigenspace. In this precise sense the nonzero spectrum closes after a single step away from $b$, and Pasch configurations are simply the combinatorial name for that step. We expect this to be an instance of a more general phenomenon, in which the nonzero spectrum of an idempotent in a Steiner-system algebra is controlled by a bounded neighbourhood of its support, with correspondingly larger local configurations governing larger systems such as projective geometries over $\mathbb F_2$. A
systematic study of these local closure phenomena lies beyond the scope of the present paper and will be pursued elsewhere.
Thus the realization problem for spectral profiles remains open even for the Grassmannian geometries $G_2(m)$, while the corresponding theory for more general Steiner systems is still largely unexplored.

\begin{example}[The determinant and the permanent]\label{rem:detper}
It is instructive to test the invariants of this section on an example
lying outside the Grassmannian family. Expanding the determinant of a
$3\times3$ matrix, take as points the nine entries $a_{rc}$ and as blocks
the six monomials
$a_{1\sigma(1)}a_{2\sigma(2)}a_{3\sigma(3)}$, $\sigma\in S_3$.
Two indices lie in a common monomial precisely when they occupy distinct rows and distinct columns, and then they lie in exactly one, so this is a
partial Steiner triple system: the configuration $(9_2,6_3)$. Identifying the nine points with $AG(2,3)=\mathbb F_3\times\mathbb F_3$ by
$a_{rc}\leftrightarrow(r,c)$, the six blocks are the lines of the two
parallel classes of nonzero finite slope, $y=x+b$ and $y=2x+b$
($b\in\mathbb F_3$); the former class consists of the three even
permutations of $S_3$ and the latter of the three odd ones.
The incidence algebra determined by the trivial decoration yields the
permanent, while the decorated incidence algebra of the decoration
$\delta(b_\sigma)=\operatorname{sgn}(\sigma)$ yields the determinant.

Three features are relevant here. First, this geometry contains
\emph{no} Pasch configurations at all: any four blocks include two from
the same parallel class, which are disjoint. Hence, in the notation \eqref{d0d1d2}, $d_2=0$,  $\ker d_2/\operatorname{im}d_1\cong\mathbb F_2$, the nontrivial class being represented by $\operatorname{sgn}$. Indeed the
three even monomials together use each of the nine entries exactly once,
and so do the three odd ones, so that
$\prod_{\text{even}}\delta^\varepsilon=\prod_{\text{odd}}\delta^\varepsilon$
for every gauge transformation $\varepsilon$, while $\operatorname{sgn}$
gives $+1$ and $-1$. Thus determinant and permanent \textit{are not gauge
equivalent}, and they exhaust the two gauge classes. In contrast to
$G_2(m)$, the Pasch-sign pattern is here not merely insufficient to
separate them: it is vacuous.
Second, the two gauge classes are not identified by any automorphism of the
configuration, since the two algebras are not isomorphic. In the notation
of \eqref{chain} this gives $N=2$ while $a=1$, both algebras having the
same block spectral profile. The inequality $a\le N$ of \eqref{chain} is
therefore strict already in this smallest of examples, and the collapse
observed for $G_2(7)$ is not a large-$m$ phenomenon.

Third, the distinction is recovered as soon as one leaves the block
idempotents.
The contrast between the two decorations of $(9_2,6_3)$ is already visible
in their idempotents.
In the determinant algebra the nonzero idempotents are a single $3$-dimensional orbit of the orthogonal group acting by algebra automorphisms with trivial stabilizers, a manifestation of the strong spectral rigidity of Hsiang algebras.
In the permanent algebra the variety of nonzero idempotents has several components, and the spectrum is no longer constant along it.

\begin{center}
\begin{tabular}{l|l|l}
Algebra & Idempotents
& Peirce spectrum\\
\midrule
$\det$
 & $\tfrac12O$, $O\in O(3)$ \ (dim $3$)
 & $\{1,(\tfrac12)^3,(-\tfrac12)^5\}$\\
\midrule
$\operatorname{per}$
 & $\tfrac14uv^{T}$, $u,v\in\{\pm1\}^3$ \ ($16$ isolated)
 & $\{1,(\tfrac14)^4,(-\tfrac12)^4\}$\\
 & $\tfrac12P$, $P$ a signed permutation
 & $\{1,(\tfrac12)^3,(-\tfrac12)^5\}$\\
 &(block idempotents)\\
 & circles through the previous ones
 & $\{1,\tfrac12,(-\tfrac12)^3,\pm s,\pm t\}$,\\
 && where $s^2+t^2=\tfrac12$
\end{tabular}
\end{center}

The circles are given explicitly by
$$
c=
\begin{pmatrix}
0&a&b\\ \tfrac12&0&0\\ 0&b&a
\end{pmatrix},
\qquad
a^2+b^2=\tfrac14,
$$
together with their images under the symmetries of the configuration;
along such a circle the two free eigenvalues satisfy
$s^2+t^2=\tfrac12$ and vary continuously, so that the permanent algebra
has infinitely many distinct idempotent spectra. The block idempotents
are precisely the points $s=t=\tfrac12$ of these circles, where several
of them cross; they are singular points of the idempotent variety.

Two consequences are worth recording. First, the block idempotents of the
two algebras have the \emph{same} Peirce spectrum, so no invariant built
from block idempotents can separate the determinant from the permanent.
Second, what does separate them lives on the full idempotent variety: one
spectrum in the first case, a one-parameter family of spectra in the
second.
\end{example}


\subsection*{Acknowledgment}
V.G.T. has been supported by the Stiftelsen L\"angmanska kulturfonden, Grant BA26-2924.
Part of this work was presented by the second author (V.G.T.) at the 14th International Conference on Clifford Algebras and Their Applications in Mathematical Physics (ICCA14). He thanks the organizers for the invitation and hospitality, and for providing an excellent opportunity to discuss this work with the conference participants.


\def\polhk#1{\setbox0=\hbox{#1}{\ooalign{\hidewidth
  \lower1.5ex\hbox{`}\hidewidth\crcr\unhbox0}}} \def\cprime{$'$}
  \def\cprime{$'$} \def\cprime{$'$}
  \def\polhk#1{\setbox0=\hbox{#1}{\ooalign{\hidewidth
  \lower1.5ex\hbox{`}\hidewidth\crcr\unhbox0}}} \def\cprime{$'$}
  \def\cprime{$'$} \def\cprime{$'$} \def\cprime{$'$}
  \def\polhk#1{\setbox0=\hbox{#1}{\ooalign{\hidewidth
  \lower1.5ex\hbox{`}\hidewidth\crcr\unhbox0}}} \def\cprime{$'$}
  \def\Dbar{\leavevmode\lower.6ex\hbox to 0pt{\hskip-.23ex \accent"16\hss}D}
  \def\cprime{$'$} \def\cprime{$'$} \def\cprime{$'$} \def\cprime{$'$}
  \def\cprime{$'$} \def\cprime{$'$} \def\cprime{$'$} \def\cprime{$'$}
  \def\cprime{$'$} \def\cprime{$'$} \def\cprime{$'$} \def\dbar{\leavevmode\hbox
  to 0pt{\hskip.2ex \accent"16\hss}d} \def\cprime{$'$} \def\cprime{$'$}
  \def\cprime{$'$} \def\cprime{$'$} \def\cprime{$'$} \def\cprime{$'$}
  \def\cprime{$'$} \def\cprime{$'$} \def\cprime{$'$} \def\cprime{$'$}
  \def\cprime{$'$} \def\cprime{$'$} \def\cprime{$'$} \def\cprime{$'$}
  \def\cprime{$'$} \def\cprime{$'$} \def\cprime{$'$} \def\cprime{$'$}
  \def\cprime{$'$} \def\cprime{$'$} \def\cprime{$'$} \def\cprime{$'$}
  \def\cprime{$'$} \def\cprime{$'$} \def\cprime{$'$} \def\cprime{$'$}
  \def\cprime{$'$} \def\cprime{$'$} \def\cprime{$'$} \def\cprime{$'$}
  \def\cprime{$'$} \def\cprime{$'$}
\providecommand{\bysame}{\leavevmode\hbox to3em{\hrulefill}\thinspace}
\providecommand{\MR}{\relax\ifhmode\unskip\space\fi MR }
\providecommand{\MRhref}[2]{%
  \href{http://www.ams.org/mathscinet-getitem?mr=#1}{#2}
}
\providecommand{\href}[2]{#2}

\end{document}